\documentclass[11pt]{amsart}
\usepackage[usenames]{color}
\usepackage{fullpage}
\usepackage{amscd}
\usepackage{amssymb, latexsym}
\usepackage[all,2cell,ps]{xy}
\usepackage{mathdots}

\theoremstyle{plain}
\newtheorem{thm}{Theorem}[section]
\newtheorem{theorem}[thm]{Theorem}

\newtheorem{corollary}[thm]{Corollary}

\newtheorem{lemma}[thm]{Lemma}

\newtheorem{proposition}[thm]{Proposition}
\newtheorem{ques}[thm]{Question}
\newtheorem{fact}[thm]{Fact}
\newtheorem{obser}[thm]{Observation}
\newtheorem{conjecture}[thm]{Conjecture}

\theoremstyle{definition}
\newtheorem{de}[thm]{Definition}

\newtheorem{rem}[thm]{Remark}
\newtheorem{remark}[thm]{Remark}
\newtheorem{example}[thm]{Example}
\newtheorem{convention}[thm]{Convention}
\newtheorem{algorithm}[thm]{Algorithm}

\newcommand{\Z}{\mathbb{Z}}

\newcommand{\id}{\mathrm{id}}

\newcommand{\lmlt}[1]{\mathrm{LMlt}(#1)}

\newcommand{\ld}{\backslash}

\newcommand{\Aut}{\mathop{\mathrm{Aut}}}

\newcommand{\LMlt}{\mathop{\mathrm{LMlt}}}

\newcommand{\eqrel}[2]{\mathrel{\stackrel{\scriptstyle\eqref{#1}}{#2}}}
\numberwithin{equation}{section}

\begin{document}

\title{The construction of multipermutation solutions of the Yang-Baxter equation of level 2}

\author{P\v remysl Jedli\v cka}
\author{Agata Pilitowska}
\author{Anna Zamojska-Dzienio}

\address{(P.J.) Department of Mathematics, Faculty of Engineering, Czech University of Life Sciences, Kam\'yck\'a 129, 16521 Praha 6, Czech Republic}
\address{(A.P., A.Z.) Faculty of Mathematics and Information Science, Warsaw University of Technology, Koszykowa 75, 00-662 Warsaw, Poland}

\email{(P.J.) jedlickap@tf.czu.cz}
\email{(A.P.) A.Pilitowska@mini.pw.edu.pl}
\email{(A.Z.) A.Zamojska-Dzienio@mini.pw.edu.pl}

\keywords{Yang-Baxter equation, set-theoretic solution, multipermutation solution, one-sided quasigroup, birack, left distributivity, rack.}
\subjclass[2010]{Primary: 16T25. Secondary: 20N02, 20B25.}

\date{\today}

\begin{abstract}
We study involutive set-theoretic solutions of the Yang-Baxter equation of multipermutation level 2. These solutions happen to fall into two classes -- distributive ones and non-distributive ones. The distributive ones can be effectively constructed using a set of abelian groups and a matrix of constants. Using this construction, we enumerate all distributive involutive solutions up to size 14. The non-distributive solutions can be also easily constructed, using a distributive solution and a permutation.
\end{abstract}

\maketitle
\section{Introduction}

The Yang-Baxter equation is a fundamental equation occurring in integrable models in statistical mechanics and quantum field theory~\cite{Jimbo}.
Let $V$ be a vector space. A {\em solution of the Yang--Baxter equation} is a linear mapping $r:V\otimes V\to V\otimes V$ such that
\[
(id\otimes r) (r\otimes id) (id\otimes r)=(r\otimes id) (id\otimes r) (r\otimes id).
\]
Description of all possible solutions seems to be extremely difficult and therefore
there were some simplifications introduced (see e.g. \cite{Dr90}).

Let $X$ be a basis of the space $V$ and let $\sigma:X^2\to X$ and $\tau: X^2\to X$ be two mappings.
We say that $(X,\sigma,\tau)$ is a {\em set-theoretic solution of the Yang--Baxter equation} if
the mapping $x\otimes y \mapsto \sigma(x,y)\otimes \tau(x,y)$ extends to a solution of the Yang--Baxter
equation. It means that $r\colon X^2\to X^2$, where $r=(\sigma,\tau)$ satisfies the \emph{braid relation}:

\begin{equation}\label{eq:braid}
(id\times r)(r\times id)(id\times r)=(r\times id)(id\times r)(r\times id).
\end{equation}

A solution is called {\em non-degenerate} if the mappings $\sigma(x,\_)$ and $\tau(\_\,,y)$ are bijections,
for all $x,y\in X$.
A solution $(X,\sigma,\tau)$ is {\em involutive} if $r^2=\mathrm{id}_{X^2}$, and it is
\emph{square free} if $r(x,x)=(x,x)$, for every $x\in X$.

\begin{convention}

All solutions, we study in this paper, are set-theoretic, non-degenerate and involutive so we will call them simply \emph{solutions}. The set $X$ can be of arbitrary cardinality.

\end{convention}

It is known (see e.g. \cite{S06, GIM08, D15}) that there is a one-to-one correspondence between solutions of the Yang-Baxter equation and involutive \emph{biracks} $(X,\circ,\ld_{\circ},\bullet,/_{\bullet})$ -- algebras which have a structure of two one-sided quasigroups $(X,\circ,\ld_{\circ})$ and $(X,\bullet,/_{\bullet})$ and satisfy some additional identities \eqref{eq:b1}--\eqref{eq:rinv}. This fact allows one to characterize solutions of the Yang-Baxter equation applying universal algebra tools.

In \cite[Section 3.2]{ESS} Etingof, Schedler and Soloviev introduced, for each solution $(X,\sigma,\tau)$, the equivalence relation $\sim$ on the set $X$: for each $x,y\in X$

\[
x\sim y\quad \Leftrightarrow\quad \tau(\_\,,x)=\tau(\_\,,y).
\]
They showed that the quotient set $X/\mathord{\sim}$ can be again endowed
with a structure of a solution and they call such a solution the {\em retraction} of the solution~$X$ and denote it by
$\mathrm{Ret}(X)$. A solution~$X$ is said to be a multipermutation solution of level~$k$, if $k$ is the smallest integer
such that $|\mathrm{Ret}^k(X)|=1$.
Since then many results appeared that study multipermutation solutions,
often of a small level, e.g. \cite[Section 8]{GIC12} or \cite{GIM11} which focused on the quantum spaces of (finite) solutions with multipermutation level~2.
Square-free multipermutation solutions are always decomposable~\cite{Rump} and several authors gave descriptions of some of these solutions either as a generalized twisted union~\cite{ESS,GI04,CJO10} or a strong twisted
union~\cite{GIC12,GIM08}. We have to say, however, that this approach brings decompositions only and does not offer a direct way how to construct such solutions.
In our work we bring a simple-to-use way how to construct multipermutation solutions of level~2 using abelian groups only. Moreover, our approach works for all such solutions,
not only for square-free ones.

It was proved by Gateva-Ivanova and Cameron~\cite[Proposition 8.2]{GIC12} that, for a square-free solution~$(X,\sigma,\tau)$, we have $\sigma_x=\sigma(x,\_)\in \Aut(X)$, for all $x\in X$, if and only if the solution $X$ is a multipermutation solution of level~$2$. In the language of identities this is equivalent to $(X,\circ,\ld_{\circ})$ being left distributive. It turns out, that this property can be characterized by
several different identities and the equivalence of these identities holds in more general structures. This is why we in Section~2 study left quasigroups and we establish connections between several identities of binary algebras.

Given a square-free solution of a multipermutation level~2 and the associated birack $(X,\circ,\backslash_\circ,\bullet,/_\bullet)$, the algebra $(X,\circ,\ld_{\circ})$ turns out to be a medial quandle (see Lemma \ref{lm:2red}).
The structure of medial quandles was studied in~\cite{JPSZ} and one of the main results was a construction of medial quandles based on a set of abelian groups, a matrix of homomorphisms and a matrix of constants.
In Section~3, we
adapt the construction to the current context (the matrix of homomorphisms is actually not needed here anymore) and we generalize it
so that it may include all distributive solutions, not only those square-free ones.

If a solution $(X,\sigma,\tau)$
is a multipermutation solution of level~$2$
and $e\in X$, then $(X,\sigma',\tau')$, where $\sigma'(x,y)=\sigma(x,\sigma_e^{-1}(y))$ and $\tau'(x,y)=\sigma_e(\tau(x,\sigma_e^{-1}(y)))$, is a distributive solution~ [Theorem \ref{th:2persol}].
This phenomenon is a special case of something called an {\em isotope}.
In Section~4 we study these special isotopes on the level of left quasigroups.

In Section~5 we finally get to biracks and we show what results from
Section~2 and Section~3 tell us in the world of distributive involutive biracks.
Some of these results generalize the latest results by Gateva-Ivanova~\cite{GI18}. We then translate the results into the language
of solutions of the Yang-Baxter equation in Section~7.
We prove that a solution is of multipermutation level~2 if and only if it is medial [Theorem~\ref{th:medial-char}] and
we show equivalent properties for distributive solutions of multipermutation level~2 [Theorem~\ref{th:dist-char}]. We also rephrase how to construct any solution of multipermutation level~2. Additionally, we present a short direct proof that each abelian group (of arbitrary order) is an IYB group (Theorem~\ref{thm7.11}).

In Section~6 we focus on non-distributive biracks associated
to solutions of multipermutation level~2 and the isotopy
that transforms them into distributive ones. This way we can effectively
construct all solutions of multipermutation level~2, which
is used in Section~8 to enumerate small biracks.
Distributive biracks are enumerated up to size~14, for the others of
multipermutation level~2 we give an upper bound only since we lack
an easy-to-use isomorphism criterion.

\section{Left quasigroups}\label{sec2}

In this preliminary section we introduce some identities that
we shall use throughout the text and we show a few examples
of left quasigroups with such properties.

\begin{de}
A \emph{left quasigroup} is an algebra $(X,\circ,\ld_{\circ})$ with two binary operations: the \emph{left multiplication} and the
\emph{left division} respectively, satisfying for every $x,y\in X$ the following conditions:
\begin{equation}\label{eq:lq}
 x\circ(x\ld_{\circ} y)=y=x\ld_{\circ} (x\circ y).
 \end{equation}
\end{de}
A {\em right quasigroup} is defined analogously as an algebra $(X,\bullet,/_{\bullet})$ with two binary operations of \emph{right multiplication} and the
\emph{right division}  satisfying for every $x,y\in X$ the conditions:
\begin{equation}\label{eq:rq}
  (y/_{\bullet} x)\bullet x=y=(y\bullet x)/_{\bullet} x.
 \end{equation}

Condition \eqref{eq:lq} simply means that
all \emph{left translations} $L_x\colon X\to X$ by $x$
\begin{align*}
L_x(a)= x\circ a,
\end{align*}
are bijections, with $L_x^{-1}(a)=x\ld_{\circ} a$. Equivalently, that for every $x,y\in X$, the equation $x\circ u=y$ has the unique solution $u=L_x^{-1}(y)$ in $X$. Similarly, Condition \eqref{eq:rq} gives that
all \emph{right translations} ${\textbf{\emph R}}_x\colon X\to X$ by $x$; ${\textbf{\emph R}}_x(a)= a\bullet x$,
are bijections with ${\textbf{\emph R}}_x^{-1}(a)=a/_{\bullet} x$.

\vskip 2mm
\noindent
It is obvious that if $(X,\circ,\ld_{\circ})$ is a left quasigroup then $(X,\ld_{\circ},\circ)$
  is also a left quasigroup.
\vskip 2mm
The \emph{left multiplication group} of a left quasigroup $(X,\circ,\ld_{\circ})$ is the permutation group generated by left translations,
i.e. the group $\LMlt(X)=\langle L_x:\ x\in X\rangle$.

\begin{de}
Let $m\in \mathbb{N}$. A left quasigroup $(X,\circ,\ld_{\circ})$ is called:
\begin{itemize}
\item \emph{left distributive}, if
for every $x,y,z\in X$:
\begin{align}\label{eq:left}
x\circ(y\circ z)=(x\circ y)\circ(x\circ z)\quad \Leftrightarrow\quad L_xL_y=L_{x\circ y}L_x
\end{align}
\item $m$-\emph{reductive}, if for every $x_0,x_1,x_2,\ldots,x_m\in X$:
\begin{align}\label{eq:mred}
(\ldots((x_0\circ x_1)\circ x_2)\ldots)\circ x_m=(\ldots((x_1\circ x_2)\circ x_3)\ldots)\circ x_m
\end{align}
\item $m$-\emph{permutational}, if for every $x,y,x_1,x_2,\ldots,x_m\in X$:
\begin{align}\label{eq:mper}
(\ldots((x\circ x_1)\circ x_2)\ldots)\circ x_m=(\ldots((y\circ x_1)\circ x_2)\ldots)\circ x_m
\end{align}
\item \emph{medial}, if for every $x,y,z,t\in X$:
\begin{align}\label{eq:med}
(x\circ y)\circ (z\circ t)=(x\circ z)\circ (y\circ t)\quad \Leftrightarrow\quad L_{x\circ y}L_z=L_{x\circ z}L_y
\end{align}
\item \emph{right cyclic}, if it satisfies the \emph{right cyclic} law, i.e. for every $x,y,z\in X$:
\begin{equation}\label{eq:RC}
 (x\ld_{\circ} y)\ld_{\circ}( x \ld_{\circ} z)=(y\ld_{\circ} x)\ld_{\circ}(y\ld_{\circ} z)\quad \Leftrightarrow\quad L^{-1}_{x\ld_{\circ} y}L^{-1}_x=L^{-1}_{y\ld_{\circ} x}L^{-1}_y
 \quad \Leftrightarrow\quad L_xL_{x\ld_{\circ} y}=L_yL_{y\ld_{\circ} x}
 \end{equation}
\item \emph{non-degenerate}, if the mapping
\begin{equation}\label{eq:ND}
 T\colon X\to X;\quad  x\mapsto x\ld_{\circ} x,
  \end{equation}
  is a bijection

\item \emph{idempotent}, if for every $x\in X$
\begin{align}\label{eq:idemp}
x\circ x=x\quad \Leftrightarrow\quad L_x(x)=x\quad \Leftrightarrow\quad L_x^{-1}(x)=x.
\end{align}
\end{itemize}
\end{de}
\noindent
A left distributive left quasigroup is a \emph{rack}. Idempotent racks are called \emph{quandles}.
\vskip 2mm

The condition of left distributivity is well established in the literature. It appeared in a natural way in such areas as low-dimensional topology -- in knot \cite{Car} and braid \cite{D} invariants or in the theory of symmetric spaces \cite{Loo}. Probably at first it was introduced already at the end of 19th century in papers of Peirce \cite{Pei} and Schr${\rm \ddot{o}}$der \cite{Sch}. Recently, Lebed and Vendramin \cite{LV} considered the condition in the context of solutions of the Yang-Baxter equation.

The property of mediality was first investigated as a generalization of the associative law for quasigroups (see Murdoch \cite{Mur} and Sushkievich \cite{Sus}). It appears also in the characterization of mean value functions \cite{Acz}. The first systematic approach to medial groupoids was undertaken by Je\v{z}ek and Kepka in \cite{JK}. Idempotency in the theory of the Yang-Baxter solutions is called \emph{square-freeness}. Idempotent and medial quasigroups are investigated since the middle of 20th century. In the wider context, two monographs \cite{RS85,RS} of Romanowska and Smith are devoted to idempotent and medial algebras called \emph{modes} which are present in different branches of mathematics and find
applications in computer science, economics, physics, and biology.

$2$-reductive groupoids were considered by P\l onka as a special case of \emph{cyclic groupoids} \cite{Plo}. The more general $m$-reductive modes were investigated  in \cite{PRR96} and  \cite{PR}. In \cite{JPSZ} and \cite{JPZ} $m$-reductive quandles were characterized.

Right cyclic quasigroups (under the name \emph{cycle sets}) were introduced by Rump in \cite{Rump}. He showed that there is a correspondence between solutions of the Yang-Baxter equation  and non-degenerate cycle sets (see Theorem \ref{th:rclq}).

Finally, Condition \eqref{eq:mper} was defined by Gateva-Ivanova in \cite[Remark 4.6]{GI18} to describe multipermutation solutions of the Yang-Baxter equation (see Theorem \ref{GI:mper}). Earlier, Gateva-Ivanova and Cameron used Condition \eqref{eq:mred} (see \cite[Theorem 5.15]{GIC12}). They did not name these properties.

\begin{obser}\label{ob:1}
Each $m$-reductive left quasigroup is $m$-permutational and each $m$-permutational, idempotent left quasigroup is $m$-reductive.
\end{obser}

For a solution of the Yang-Baxter equation,  Gateva-Ivanova considered in \cite[Definition 4.3]{GI18} Condition $(\ast)$ which in the language of left quasigroups $(X,\circ,\ld_{\circ})$ means that
\begin{align*}
\tag{$\ast$} \forall x\in X\quad \exists a\in X\quad  a\circ x=x.
\end{align*}
It is evident, that each idempotent left quasigroup satisfies Condition $(\ast)$.

By Observation \ref{ob:1} we have that each idempotent $m$-permutational left quasigroup is $m$-reductive, for arbitrary $m\in \mathbb{N}$. The same is also true for $m$-permutational left quasigroups which satisfy Condition $(\ast)$ (see also \cite[Proposition 4.7]{GI18}).
\begin{lemma}\label{lm:mper}
Let $(X,\circ,\ld_{\circ})$ be a left quasigroup which satisfies Condition $(\ast)$ and $m\in \mathbb{N}$. Then $(X,\circ,\ld_{\circ})$ is $m$-permutational if and only if it is $m$-reductive.
\end{lemma}
\begin{proof}
We have only to prove that each $m$-permutational left quasigroup which satisfies Condition $(\ast)$ is $m$-reductive. But it is evident. By Condition $(\ast)$ for each $x\in X$ there exists $a_x\in X$ such that $a_x\circ x=x$. Then for every $x_0,x_1,x_2,\ldots,x_m\in X$  we have:
\begin{align*}
(\ldots((x_0\circ x_1)\circ x_2)\ldots)\circ x_m\stackrel{\scriptsize \eqref{eq:mper}}=
(\ldots((a_{x_1}\circ x_1)\circ x_2)\ldots)\circ x_m=
(\ldots((x_1\circ x_2)\circ x_3)\ldots)\circ x_m,
\end{align*}
which completes the proof.
\end{proof}
In this paper we are mainly interested in $2$-reductive and $2$-permutational left quasigroups. In particular,
a left quasigroup $(X,\circ,\ld_{\circ})$ is $2$-reductive if, for every $x,y,z\in X$:
\begin{align}\label{eq:red}
(x\circ y)\circ z=y\circ z \quad \Leftrightarrow\quad L_{x\circ y}=L_y,
\end{align}
and it is $2$-permutational if for every $x,y,z,t\in X$:
\begin{align}\label{eq:2per}
(z\circ x)\circ y=(t\circ x)\circ y\quad \Leftrightarrow\quad L_{z\circ x}=L_{t\circ x}.
\end{align}

\begin{example}\label{ex:nie2red}
The left quasigroup $(\{0,1,2,3\},\circ,\ld_{\circ})$ with the following left multiplication:
\[
  \begin{array}{c|cccc}
   \circ & 0 & 1 & 2 & 3\\
   \hline
   0 & 0 & 1 & 2 &3\\
   1 & 2 & 3 & 0 & 1 \\
   2 & 0 & 1 & 2 & 3\\
   3 & 2 & 3 & 0 &1
  \end{array}
  \]
is a 2-reductive rack and, according to Lemma~\ref{lm:2red}, it is medial.
In this case $L_0=L_2=\mathrm{id}$ and $L_1=L_3=(02)(13)$.
\end{example}

\begin{example}\label{ex:nondistr}
Let $(\{0,1,2,3\},\circ,\ld_{\circ})$ be a left quasigroup with the following left multiplication:
 \[
  \begin{array}{c|cccc}
   \circ & 0 & 1 & 2 & 3\\
   \hline
   0 & 1 & 0 & 3 &2\\
   1 & 3 & 2 & 1 & 0 \\
   2 & 1 & 0 & 3 & 2\\
   3 & 3 & 2 & 1 &0
  \end{array}.
 \]
Clearly, $L_0=L_2=(01)(23)$ and $L_1=L_3=(03)(12)$. One can check that $(\{0,1,2,3\},\circ,\ld_{\circ})$ is both right cyclic and
$2$-permutational but neither left distributive nor $2$-reductive. Additionally, by Corollary \ref{cor:med}, $(\{0,1,2,3\},\circ,\ld_{\circ})$ is medial.
\end{example}

For a left quasigroup $(X,\circ,\ld_{\circ})$, Condition \eqref{eq:left} means that all left translations for every $x\in X$, are automorphisms of $(X,\circ)$, i.e.
for every $x,y,z\in X$
\begin{align}\label{eq:auto}
L_x(y\circ z)=L_x(y)\circ L_x(z).
\end{align}

\begin{lemma}\label{rem:eqival}
Let $(X,\circ,\ld_{\circ})$ be a left quasigroup. Then
\begin{itemize}
\item $(X,\circ,\ld_{\circ})$ is left distributive if and only if $(X,\ld_{\circ},\circ)$ is left distributive.
\item $(X,\circ,\ld_{\circ})$ is $2$-reductive if and only if $(X,\ld_{\circ},\circ)$ is $2$-reductive.
\item $(X,\circ,\ld_{\circ})$ is medial if and only if $(X,\ld_{\circ},\circ)$ is medial.
\item $(X,\circ,\ld_{\circ})$ is idempotent if and only if $(X,\ld_{\circ},\circ)$ is idempotent.
\end{itemize}
\end{lemma}
\begin{proof}
If $L_x$ is an automorphism, then $L_x^{-1}$
is clearly an automorphism as well, giving
 Property \eqref{eq:auto}. Furthermore, for every $x,y,z\in X$:
\begin{multline*}
(x\circ y)\circ z=y\circ z \quad \stackrel{\scriptsize y\mapsto x\ld_{\circ} y}\Longrightarrow\quad (x\circ (x\ld_{\circ} y))\circ z=(x\ld_{\circ} y)\circ z \quad \eqrel{eq:lq}\Longleftrightarrow\quad y\circ z=(x\ld_{\circ} y)\circ z\quad \Leftrightarrow\\
L_{x\ld_{\circ} y}(z)=L_y(z)\quad \Leftrightarrow\quad L_{x\ld_{\circ} y}^{-1}(z)=L_y^{-1}(z)\quad \Leftrightarrow\quad(x\ld_{\circ} y)\ld_{\circ} z=y\ld_{\circ} z\stackrel{\scriptsize y\mapsto x\circ y}\Longrightarrow\\
(x\ld_{\circ} (x\circ y))\ld_{\circ} z=(x\circ y)\ld_{\circ} z \quad \eqrel{eq:lq}\Longleftrightarrow
y\ld_{\circ} z=(x\circ y)\ld_{\circ} z\quad \Leftrightarrow\\
L_y^{-1}(z)=L_{x\circ y}^{-1}(z)\quad \Leftrightarrow\quad L_{y}(z)=L_{x\circ y}(z)\quad \Leftrightarrow\quad (x\circ y)\circ z=y\circ z.
\end{multline*}
Similarly, we can show that for every $x,y,z,t\in X$ (see also \cite[Exercise 8.6H]{RS})
\[
(x\circ y)\circ (z\circ t)=(x\circ z)\circ (y\circ t) \quad \Leftrightarrow\quad (x\ld_{\circ} y)\ld_{\circ} (z\ld_{\circ} t)=(x\ld_{\circ} z)\ld_{\circ} (y\ld_{\circ} t).\qedhere
\]
\end{proof}

Next examples show that, for right cyclic or $2$-permutational left quasigroup $(X,\circ,\ld_{\circ})$, the left quasigroup $(X,\ld_{\circ},\circ)$ does not have to be right cyclic or $2$-permutational.
\begin{example}\label{ex:non2per}
Let $(\{0,1,2\},\circ,\ld_{\circ})$ be a left quasigroup with the following left multiplication and left division:

 \[
  \begin{array}{c|ccc}
   \circ & 0 & 1 & 2 \\
   \hline
   0 & 1 & 0 & 2\\
   1 & 2 & 0 & 1 \\
   2 & 2 & 0 & 1
  \end{array}
   \qquad
  \begin{array}{c|ccc}
   \ld_{\circ} & 0 & 1 & 2 \\
   \hline
   0 & 1 & 0 & 2\\
   1 & 1& 2 & 0  \\
   2 & 1& 2 & 0
  \end{array},
 \]
or equivalently, $L_0=(01)=L_0^{-1}$, $L_1=L_2=(021)$ and $L_1^{-1}=L_2^{-1}=(012)$.
This left quasigroup is $2$-permutational, but
\[
0=0\ld_{\circ} 1=(0\ld_{\circ} 1)\ld_{\circ} 1\neq (1\ld_{\circ} 1)\ld_{\circ} 1=2\ld_{\circ} 1=2.
\]
\end{example}

\begin{example}\label{ex:noncyc}
Let $(\{0,1,2,3\},\circ,\ld_{\circ})$ be a left quasigroup with the following left multiplication and left division:
 \[
  \begin{array}{c|cccc}
   \circ & 0 & 1 & 2 & 3\\
   \hline
   0 & 0 & 1 & 3 &2\\
   1 & 2 & 3 & 1 & 0 \\
   2 & 3 & 2 & 0 & 1\\
   3 & 1 & 0 & 2 & 3
  \end{array}
   \qquad
  \begin{array}{c|cccc}
   \ld_{\circ} & 0 & 1 & 2 & 3\\
   \hline
   0 & 0 & 1 & 3 &2\\
   1 & 3 & 2 & 0 & 1 \\
   2 & 2 & 3 & 1 & 0\\
   3 & 1 & 0 & 2 & 3
  \end{array},
 \]
i.e. $L_0=(23)=L_0^{-1}$, $L_1=(0213)=L_2^{-1}$, $L_2=(0312)=L_1^{-1}$ and $L_3=(01)(23)=L_3^{-1}$.
In this case the left quasigroup is right cyclic, but
\[
2=1\circ 0=(0\circ 1)\circ(0\circ 0)\neq(1\circ 0)\circ(1\circ 0)=2\circ 2=0.
\]
\end{example}
Directly from \eqref{eq:left} and Lemma~\ref{rem:eqival}
we obtain that the left distributivity implies,
 for every $x,y\in X$,
\begin{align}\label{eq:ab}
L_{x\circ y}=L_xL_{y}L_x^{-1} \quad {\rm and}\quad L_{x\ld_{\circ} y}=L_x^{-1}L_{y}L_x.
\end{align}
Note also that, for an arbitrary automorphism $\alpha$ of $(X,\circ)$, we have
\begin{align*}
L_{\alpha(x)}(y)=\alpha(x)\circ y=\alpha(x\circ \alpha^{-1}(y))=\alpha L_{x}\alpha^{-1}(y).
\end{align*}

\section{2-reductive racks}

 It is known ~\cite[Theorem 5.15]{GIC12} that a~square-free multipermutation solutions of level~2
 is 2-reductive. It turns out that 2-reductivity has connections
 to other identities presented in Section~2.
 We study all these connections on the class of racks
 which are, after all, an interesting class itself, having
 many applications, e.g. in knot theory \cite{FR}, \cite[Chapter 5]{EN}.
 Moreover, we can apply existing tools, like a construction
 using affine meshes which is presented in the second half of this section.

\begin{lemma}\label{lem:cycabel}
Let $(X,\circ,\ld_{\circ})$ be a rack. The following conditions are equivalent:
\begin{enumerate}
\item $(X,\circ,\ld_{\circ})$ is right cyclic;
\item the group $\LMlt(X)$ is abelian;
\item $(X,\ld_{\circ},\circ)$ is right cyclic;
\item $(X,\circ,\ld_{\circ})$ is $2$-reductive.
\end{enumerate}
\end{lemma}
\begin{proof}
In a rack, by \eqref{eq:ab}, the conditions (2) and (4) are equivalent:
\begin{align*}
&L_y=L_{x\circ y}=L_xL_yL_x^{-1}\quad \Leftrightarrow\quad
L_yL_x=L_xL_y.
\end{align*}
Furthermore,
by \eqref{eq:ab} and \eqref{eq:RC} for every $x,y,z\in X$ we have:
\begin{multline*}
(x\ld_{\circ} y)\ld_{\circ}(x\ld_{\circ} z)=(y\ld_{\circ} x)\ld_{\circ}(y\ld_{\circ} z) \quad \Leftrightarrow\quad
L_{x\ld_{\circ} y}^{-1}L_x^{-1}=L_{y\ld_{\circ} x}^{-1}L_y^{-1} \quad \Leftrightarrow\quad
L_{x}L_{x\ld_{\circ} y}=L_yL_{y\ld_{\circ} x}\\
\Leftrightarrow\quad
L_{x}L_x^{-1}L_yL_x=L_{y}L_y^{-1}L_xL_y\quad \Leftrightarrow\quad
L_{x}L_{y}=L_yL_{x}\quad \Leftrightarrow\quad
L_{x}L_{y}L_{x}^{-1}L_x=L_yL_{x}L_{y}^{-1}L_y\\
\Leftrightarrow \quad
L_{x\circ y}L_{x}=L_{y\circ x}L_{y}\quad \Leftrightarrow\quad
(x\circ y)\circ(x\circ z)=(y\circ x)\circ(y\circ z),
\end{multline*}
which completes the proof.
\end{proof}

\begin{corollary}\label{cor:2red}
Let $(X,\circ,\ld_{\circ})$ be a right cyclic left quasigroup.
Then the following conditions are equivalent:
\begin{enumerate}
\item $(X,\circ,\ld_{\circ})$ is a rack;
\item $(X,\circ,\ld_{\circ})$ is $2$-reductive.
\end{enumerate}
\end{corollary}
\begin{proof}
If $(X,\circ,\ld_{\circ})$ is a rack then
by Lemma \ref{lem:cycabel} it is 2-reductive.

Conversely, by 2-reductivity of the right cyclic left quasigroup $(X,\circ,\ld_{\circ})$ we have
\begin{align*}
L_xL_{x\ld_{\circ} y}=L_yL_{y\ld_{\circ} x}\quad  \Rightarrow\quad  L_xL_y=L_yL_x=L_{x \circ y}L_{x},
\end{align*}
which shows that $(X,\circ,\ld_{\circ})$ is left distributive.
\end{proof}

\begin{lemma}\label{lm:2red}
Let $(X,\circ,\ld_{\circ})$ be a $2$-reductive left quasigroup. Then the following conditions are equivalent:
\begin{enumerate}
\item $(X,\circ,\ld_{\circ})$ is left distributive;
\item the group $\LMlt(X)$ is abelian;
\item $(X,\circ,\ld_{\circ})$ is right cyclic;
\item $(X,\ld_{\circ},\circ)$ is right cyclic;
\item $(X,\circ,\ld_{\circ})$ is medial.
\end{enumerate}
\end{lemma}
\begin{proof}
Let $(X,\circ,\ld_{\circ})$ be a 2-reductive left quasigroup.
The implications: $(1)\Rightarrow(2)$, $(1)\Rightarrow(3)$ and $(1)\Rightarrow(4)$
directly follow by Lemma \ref{lem:cycabel}.

If the group $\LMlt(X)$ is abelian then
\begin{align*}
L_{x\circ y}L_x=L_yL_x=L_xL_y,
\end{align*}
which gives $(2)\Rightarrow(1)$.

If $(X,\circ,\ld_{\circ})$ is right cyclic then
\begin{align*}
L_{x\circ y}L_x=L_yL_x=L_yL_{y\ld_{\circ} x}=L_xL_{x\ld_{\circ} y}=L_xL_{y},
\end{align*}
and similarly, if $(X,\ld_{\circ},\circ)$ is right cyclic then
\begin{align*}
L_{x\circ y}L_x=L_{y\circ x}L_y=L_xL_{y},
\end{align*}
so  $(3)\Rightarrow(1)$ and $(4)\Rightarrow(1)$ are proved.

Finally,
\begin{align*}
&L_{x\circ y}L_z=L_{x\circ z}L_y \quad \Leftrightarrow\quad L_yL_{z}=L_zL_y\quad \Leftrightarrow\quad L_{z\circ y}L_z=L_zL_y,
\end{align*}
which shows that $(1)\Leftrightarrow(5)$ and completes the proof.
\end{proof}

In  \cite[Theorem 3.14]{JPSZ} David Stanovsk\'y and the authors
of this paper presented a general construction of medial quandles.
It turned out \cite[Theorem 6.9]{JPSZ} that the case of
2-reductive quandles is actually much less complicated
because 2-reductive quandles are rather combinatorial
than algebraic structures. Moreover, the construction
of 2-reductive quandles can be easily generalized for
2-reductive racks, as we shall see below.

\begin{de}
A \emph{trivial affine mesh} over a non-empty set $I$ is the pair
$$\mathcal A=((A_i)_{i\in I},\,(c_{i,j})_{i,j\in I}),$$ where $A_i$ are abelian groups and $c_{i,j}\in A_j$ constants such that
$A_j=\left<\{c_{i,j}\mid i\in I\}\right>$, for every $j\in I$.
\end{de}

If $I$ is a finite set we will usually display a trivial affine mesh as a pair $((A_i)_{i\in I},C)$, where $C=\,(c_{i,j})_{i,j\in I}$ is an $\left|I\right|\times \left|I\right|$ matrix.

\begin{de}
The \emph{sum of a trivial affine mesh} $\mathcal A=((A_i)_{i\in I},\,(c_{i,j})_{i,j\in I})$ over a set $I$
is an algebra $(\bigcup\limits_{i\in I} A_i,\circ,\ld_{\circ})$ defined on the disjoint union of the sets $A_i$,
with two operations
\begin{align*}
a\circ b&=b+c_{i,j},\\
a\ld_{\circ} b&=b-c_{i,j},
\end{align*}
 for every $a\in A_i$ and $b\in A_j$.
\end{de}

\begin{theorem}\label{th:tri}
An algebra $(X,\circ,\ld_{\circ})$ is a $2$-reductive rack if and only if it is the sum of some trivial affine mesh.
The orbits of the action of $\lmlt{X}$ then coincide with the groups of the mesh.
\end{theorem}
\begin{proof}
At first we show that the sum of a trivial affine mesh is a 2-reductive rack with orbits $A_i$, $i\in I$.

Let $a\in A_i$, $b\in A_j$, $c\in A_k$. Obviously
the equation $a\circ x=x+ c_{i,j}=b$ has a unique solution $x=b-c_{i,j}\in A_j$. Furthermore,
\begin{align*}
&(a\circ b)\circ c=(b+c_{i,j})\circ c=c+c_{j,k}=b\circ c,
\end{align*}
and
\begin{align*}
&a\circ (b\circ c)=a\circ (c+c_{j,k})=(c+c_{j,k})+c_{i,k}=(c+c_{i,k})+c_{j,k}=\\
&(b+c_{i,j})\circ(c+c_{i,k})=(a\circ b)\circ (a\circ c).
\end{align*}

For $x\in A_j$ and $a\in A_k$ we have
\[
L_a(x)=a\circ x=x+ c_{k,j}\in A_j.
\]
Thus the group $\LMlt(X)$ acts transitively on $A_j$ if and only if the elements $c_{k,j}$, $k\in I$, generate the group $A_j$.
\vskip 3mm
Now let $(X,\circ,\ld_{\circ})$ be a 2-reductive rack, and choose a transversal $E$ to the orbit decomposition.
By Lemma \ref{lm:2red}, the group $\LMlt(X)$ is abelian. Hence for every $e\in E$, the orbit $Xe=\{\alpha(e)\mid \alpha \in \LMlt(X)\}$ is an abelian group $(Xe,+,-,e)$ with $\alpha(e)+\beta(e)=\alpha\beta(e)$ and $-\alpha(e)=\alpha^{-1}(e)$, for $\alpha,\beta\in \LMlt(X)$.

Let for every $e,f\in E$
\[
c_{e,f}:=e\circ f=L_e(f)\in Xf.
\]

Since $\LMlt(X)$ is abelian, and each $\alpha\in \LMlt(X)$ is an automorphism of $(X,\circ)$, we have
 $\alpha(e)\circ f=L_{\alpha(e)}(f)=\alpha L_e\alpha^{-1}(f)=L_e(f)=e\circ f$. This implies that the set
\[
\{c_{e,f}\mid e\in E\}=\{e\circ f\mid e\in E\}=\{\alpha(e)\circ f\mid \alpha\in \LMlt(X), e\in E\}=\{L_a(f)\mid a\in X\}
\]
generates the group $(Xf,+,-,f)$.
This shows that $(X,\circ,\ld_{\circ})$ is
the sum of the trivial affine mesh $((Xe)_{e\in E},\,(c_{e,f})_{e,f\in E})$ over the set $E$.

Finally, let $a=\alpha(e)\in Xe$ and $b=\beta(f)\in Xf$ with $\alpha, \beta\in \LMlt(X)$. Therefore we obtain
\[
a\circ b=L_a(b)=L_{\alpha(e)}\beta(f)=L_e\beta(f)=L_e(f)+\beta(f)=c_{e,f}+b.
\]
So we verified that the sum of $((Xe)_{e\in E},\,(c_{e,f})_{e,f\in E})$ yields the original rack $(X,\circ,\ld_{\circ})$.
\end{proof}

Note that the sum of such trivial affine mesh is idempotent if and only if $c_{i,i}=0$, for each $i\in I$.

\begin{theorem}\label{th:iso}
Let  $\mathcal A=((A_i)_{i\in I},\,(a_{i,j})_{i,j\in I})$ and  $\mathcal B=((B_i)_{i\in I},\,(b_{i,j})_{i,j\in I})$ be two trivial affine meshes, over the same index set $I$.
Then the sums of $\mathcal A$ and $\mathcal B$ are isomorphic $2$-reductive racks if and only if there is a bijection $\pi$ of the set $I$ and group isomorphisms $\psi_i\colon A_i\to B_{\pi (i)}$ such that $\psi_j(a_{i,j})=b_{\pi (i),\pi (j)}$, for every $i,j\in I$.
\end{theorem}
\begin{proof}
The proof goes in the same way as the proof of \cite[Theorem 4.2]{JPSZ} for medial quandles in the case of 2-reductive ones.
\end{proof}

\begin{example}
Up to isomorphism, there are exactly five 2-reductive racks of size 3. They are the sums of the following trivial affine meshes:
\begin{itemize}
    \item One orbit: $((\Z_3),(1)).$

        \item Two orbits: $((\Z_2,\Z_1),\left(\begin{smallmatrix}1&0\\0&0\end{smallmatrix}\right))$, $((\Z_2,\Z_1),\left(\begin{smallmatrix}0&0\\1&0\end{smallmatrix}\right))$ and $((\Z_2,\Z_1),\left(\begin{smallmatrix}1&0\\1&0\end{smallmatrix}\right))$.

        \item Three orbits: $((\Z_1,\Z_1,\Z_1),\,\left(\begin{smallmatrix}0&0&0\\0&0&0\\0&0&0\end{smallmatrix}\right))$.
\end{itemize}
\end{example}

Theorems \ref{th:tri} and \ref{th:iso} allow us to enumerate $2$-reductive racks, up to isomorphism. The numbers are presented in Table \ref{Fig:count_medial} in Section \ref{sec:enum}.

\section{2-permutational left quasigroups}\label{sec4a}

Our goal in this paper is to study multipermutation
solutions of level~2. In the language of identities they are
2-permutational, see Theorem~\ref{GI:mper}.
This is why we focus on 2-permutational left quasigroups.
In particular, we link them via a permutation to 2-reductive
left quasigroups studied in the previous section. We start with a few auxiliary lemmas.

\begin{lemma}\label{lm:a5}
Let $(X,\circ,\ld_{\circ})$ be a $2$-permutational left quasigroup. Then for every $x,y,z\in X$
\begin{align}
L_{L_yL_z^{-1}(x)}=L_x.
\end{align}
\end{lemma}
\begin{proof}
By \eqref{eq:2per} and \eqref{eq:lq}, we have
\[
L_{L_yL_z^{-1}(x)}(t)=(y\circ(z\ld_{\circ} x))\circ t=(z\circ(z\ld_{\circ} x))\circ t=x\circ t=L_x(t). \qedhere
\]
\end{proof}

\begin{lemma}\label{lm:a3}
Let $(X,\circ,\ld_{\circ})$ be a medial left quasigroup. Then for every $x,y,z\in X$
\begin{enumerate}
\item $L_{z\circ x}=L_{z\circ z}L_xL_z^{-1}$;
\item $L_{z\ld_{\circ} x}=L_{z\circ z}^{-1}L_xL_z$;
\item $L_xL_z^{-1}L_y=L_yL_z^{-1}L_x$.
\end{enumerate}
\end{lemma}
\begin{proof}
Directly by mediality we have
\[
L_{z\circ x}L_z= L_{z\circ z}L_x\quad {\rm and}\quad L_{z\circ z}L_{z\ld_{\circ} x}=L_{z\circ (z\ld_{\circ} x)}L_{z}=L_xL_z.
\]
Further, by \eqref{eq:med}
\begin{align*}
L_{z\circ z}L_xL_z^{-1}L_y=
L_{z\circ x}L_y=L_{z\circ y}L_x=
L_{z\circ y}L_zL_z^{-1}L_x=L_{z\circ z}L_yL_z^{-1}L_x,
\end{align*}
which implies
\[
L_xL_z^{-1}L_y=L_yL_z^{-1}L_x. \qedhere
\]
\end{proof}
As we noticed in Examples \ref{ex:non2per} and \ref{ex:noncyc}, for right cyclic or $2$-permutational left quasigroups $(X,\circ,\ld_{\circ})$, the left quasigroup $(X,\ld_{\circ},\circ)$ need not be right cyclic nor $2$-permutational. But under some additional assumptions, they are.
\begin{lemma}\label{lm:a2}
Let $(X,\circ,\ld_{\circ})$ be a $2$-permutational medial left quasigroup. Then both left quasigroups $(X,\circ,\ld_{\circ})$ and $(X,\ld_{\circ},\circ)$ are right cyclic.
\end{lemma}
\begin{proof}
Let $x,y\in X$. Then
\begin{align*}
L_xL_{x\ld_{\circ} y}&\stackrel{\scriptsize \eqref{eq:lq}}=L_{y\circ (y\ld_{\circ} x)}L_{x\ld_{\circ} y}\stackrel{\scriptsize \eqref{eq:med}}=L_{y\circ (x\ld_{\circ} y)}L_{y\ld_{\circ} x}\stackrel{\scriptsize \eqref{eq:2per}}=L_{x\circ (x\ld_{\circ} y)}L_{y\ld_{\circ} x}\stackrel{\scriptsize \eqref{eq:lq}}=L_yL_{y\ld_{\circ} x} \quad {\rm and}\\
L_{x\circ y}L_x&\stackrel{\scriptsize \eqref{eq:lq}}=L_{x\circ y}L_{y\ld_{\circ} (y\circ x)}
\stackrel{\scriptsize \eqref{eq:med}}=L_{x\circ (y\ld_{\circ} (y\circ x))}L_{y}\stackrel{\scriptsize \eqref{eq:2per}}=L_{y\circ (y\ld_{\circ} (y\circ x))}L_{y}\stackrel{\scriptsize \eqref{eq:lq}}=L_{y\circ x}L_{y}.\qedhere
\end{align*}
\end{proof}

\begin{lemma}\label{lm:medial-2perm}
Let $(X,\circ,\ld_{\circ})$ be a right cyclic medial left quasigroup. Then both $(X,\circ,\ld_{\circ})$
and $(X,\ld_{\circ},\circ)$ are $2$-permutational.
\end{lemma}
\begin{proof}
We prove the claim first for $(X,\ld_\circ,\circ)$
using Lemma~\ref{rem:eqival}. Note that Condition \eqref{eq:med} for $(X,\ld_\circ,\circ)$ means that for $x,y,z\in X$
\begin{align}\label{eq:med_inv}
L_{x\ld_\circ y}^{-1}L_z^{-1}=L_{x\ld_\circ z}^{-1}L_y^{-1}\quad \Leftrightarrow\quad L^{-1}_{x\ld_\circ y}L_z^{-1}L_y=L^{-1}_{x\ld_\circ z}.
\end{align}
Hence,
\[L_{y\ld_\circ x}^{-1} \eqrel{eq:RC}= L_{x\ld_\circ y}^{-1}L_x^{-1}L_y
 \eqrel{eq:med_inv}=L_{x\ld_\circ x}^{-1}
\]
and the right-hand side does not depend on~$y$.
Now, for $(X,\circ,\ld_\circ)$, we notice that substituting
$y\mapsto x\ld_\circ y$ in \eqref{eq:med} we get
$L_{x\circ z}=L_yL_zL_{x\ld_\circ y}^{-1}$ which we use in
\[
 L_{x\circ y} \eqrel{eq:med}= L_{x\circ z}L_yL_z^{-1}
 = L_yL_zL_{x\ld_\circ y}^{-1}L_yL_z^{-1}
 = L_yL_zL_{y\ld_\circ y}^{-1}L_yL_z^{-1},
\]
where the last equality follows from $(X,\ld_\circ,\circ)$
being $2$-permutational. Again, the right-hand side does not depend on~$x$, which finishes the proof.
\end{proof}

In the theory of quasigroups (see e.g. \cite[Section II.2]{P}), there
is a standard method, called {\em isotopy}, how to derive a quasigroup from another quasigroup. We do not need this
notion in the full generality, we shall present
here a special case only.

\begin{de}\label{rem:iso}
Let $(X,\circ,\ld_{\circ})$ be a left quasigroup and $\pi$ be a bijection of the set $X$. Define on the set $X$ new binary operations:
\begin{align}
&x\ast y:= x\circ\pi(y)=L_x\pi (y)\quad {\rm and}\label{eq:iso1}\\
&x\ld_{\ast} y:= \pi^{-1}(x\ld_{\circ} y)=\pi^{-1}L_x^{-1}(y). \label{eq:iso2}
\end{align}
The algebra $(X,*,\ld_*)$ is called the {\em $\pi$-isotope} of
$(X,\circ,\ld_\circ)$.
\end{de}

\begin{remark}
It is easy to note that
\begin{align*}
&x\ast(x\ld_{\ast} y)=L_x\pi \pi^{-1}L_x^{-1}(y)=y \quad {\rm and}\\
&x\ld_{\ast}(x\ast y)=\pi^{-1}L_x^{-1}L_x\pi (y)=y.
\end{align*}
Therefore $(X,\ast,\ld_{\ast})$ is also a left quasigroup.
To obtain the multiplication table of~$\ast$ for a $\pi$-isotope of a \emph{finite} left quasigroup $(X,\circ,\ld_{\circ})$, one should permute all columns of the multiplication
table of $\circ$ using the permutation~$\pi$.

\end{remark}

\begin{remark}\label{rem:nd}
Let $(X,\circ,\ld_{\circ})$ be a non-degenerate left quasigroup.
It means that the mapping
\[
T\colon X\to X; \quad x\mapsto L_x^{-1}(x)=x\ld_{\circ} x,
\]
is a bijection.  If $\pi$ is a bijection of the set $X$ then the mapping
\[
T_{\pi}\colon X\to X; \quad x\mapsto \pi^{-1}L_x^{-1}(x)=x\ld_{\ast} x,
\]
is a bijection, too. This proves that the left quasigroup $(X,\ast,\ld_{\ast})$, being the $\pi$-isotope of $(X,\circ,\ld_{\circ})$, is non-degenerate.
\end{remark}

\begin{lemma}\label{lm:3cond}
Let $(X,\circ,\ld_{\circ})$ be a left quasigroup and $\pi$ be a bijection of the set $X$. Then the $\pi$-isotope of $(X,\circ,\ld_{\circ})$ is
\begin{enumerate}
\item  $2$-reductive if and only if, for every $x,y\in X$,
\begin{align}\label{eq:isoper1}
L_{L_x\pi (y)}=L_y,
\end{align}
\item $2$-permutational if and only if, for every $x,y,z\in X$,
\begin{align}\label{eq:isoper3}
L_{L_x\pi (z)}=L_{L_y\pi (z)},
\end{align}
\item left distributive if and only if, for every $x,y,z\in X$,
\begin{align}\label{eq:isoper4}
L_{L_x\pi(y)}\pi L_x=L_x\pi L_y.
\end{align}
\end{enumerate}
\end{lemma}
\begin{proof}
Let $(X,\ast,\ld_{\ast})$ be the $\pi$-isotope of $(X,\circ,\ld_{\circ})$. Hence for every $x,y,z\in X$
\begin{align*}
&(x\ast y)\ast z=y\ast z \quad \Leftrightarrow\quad    L_{L_x\pi (y)}\pi(z)=L_y\pi(z),\\
&(x\ast z)\ast t=(y\ast z)\ast t \quad \Leftrightarrow\quad    L_{L_x\pi (z)}\pi(t)=L_{L_y\pi (z)}\pi(t),\\
&(x\ast y)\ast (x\ast z)=x\ast(y\ast z)\quad \Leftrightarrow\quad
L_{x\ast y} \pi L_x\pi(z)=
L_{L_x\pi(y)}\pi L_x\pi (z)=L_x\pi L_y\pi(z).
\end{align*}
\end{proof}

\begin{corollary}
Let $(X,\circ,\ld_{\circ})$ be a left quasigroup and $\pi$ be a bijection of the set $X$. Then $(X,\circ,\ld_{\circ})$
is $2$-permutational if and only if the $\pi$-isotope of $(X,\circ,\ld_{\circ})$ is $2$-permutational.
\end{corollary}
\begin{proof}
Clearly, Condition \eqref{eq:2per} implies Condition \eqref{eq:isoper3}, for each permutation $\pi$ of the set $X$. Further, by
Lemma \ref{lm:3cond}(2) it is sufficient to show that left quasigroup $(X,\circ,\ld_{\circ})$ which satisfies Condition \eqref{eq:isoper3}
is $2$-permutational. Indeed, for every $x,y,z\in X$ we have
\begin{align*}
L_{L_x\pi (z)}=L_{L_y\pi (z)}\stackrel{\scriptsize z\mapsto \pi^{-1}(z)}\Rightarrow L_{L_x(z)}=L_{L_y(z)}.
\end{align*}
\end{proof}

\begin{corollary}\label{lm:rozdz2per1}
Let $(X,\circ,\ld_{\circ})$ be a $2$-reductive left quasigroup and $\varrho$ be a bijection on the set $X$.
Then the $\varrho$-isotope of $(X,\circ,\ld_{\circ})$ is a $2$-permutational left quasigroup.
\end{corollary}
In general $\pi$-isotope of $2$-reductive left quasigroup does not have to be $2$-reductive. The left quasigroup from Example
\ref{ex:nondistr} is the $(01)(23)$-isotope of the $2$-reductive left quasigroup from Example \ref{ex:nie2red} but it is not $2$-reductive.
\vskip 2mm

The idea of the next theorem is the following: we already know how to construct 2-reductive racks, using the construction from Section~3. Now, according to Corollary~\ref{lm:rozdz2per1},
$\varrho$-isotopes of these 2-reductive racks are 2-permutational.
And we want these $\varrho$-isotopes to be right cyclic.

\begin{theorem}\label{th:rozdz2per}
Let $(X,\circ,\ld_{\circ})$ be a $2$-reductive left quasigroup and $\varrho$ be a bijection on the set $X$ such that for every $x,y\in X$
\begin{align}
& L_{\varrho(y)}\varrho L_x=L_{\varrho(x)} \varrho L_y \quad \Leftrightarrow\quad \forall z\in X\quad \varrho(y)\circ \varrho(x\circ z)=\varrho(x)\circ\varrho(y\circ z). \label{eq:sigma}
\end{align}
Then the $\varrho$-isotope of $(X,\circ,\ld_{\circ})$ is a $2$-permutational right cyclic left quasigroup.
\end{theorem}
\begin{proof}
Let $(X,\ast,\ld_{\ast})$ be the $\varrho$-isotope of $(X,\circ,\ld_{\circ})$.
By Corollary \ref{lm:rozdz2per1}, $(X,\ast,\ld_{\ast})$ is $2$-permutational. Further, Condition \eqref{eq:sigma} is equivalent to the following one:
\begin{align}
&L^{-1}_x\varrho^{-1} L^{-1}_{\varrho(y)}=L_y^{-1}\varrho^{-1} L^{-1}_{\varrho(x)} \label{eq:sigma1}.
\end{align}
Substituting $x$ by $\varrho^{-1}(x)$ and $y$ by $\varrho^{-1}(y)$ in \eqref{eq:sigma1} we obtain:
\begin{align}\label{eq:tau}
&L^{-1}_{\varrho^{-1}(x)}\varrho^{-1} L^{-1}_y=L_{\varrho^{-1}(y)}^{-1}\varrho^{-1} L^{-1}_x.
\end{align}
Together with Lemma \ref{rem:eqival} this implies that for $x,y,z\in X$
\begin{align*}
&(x\ld_{\ast} y)\ld_{\ast}(x\ld_{\ast} z)=\varrho^{-1} L^{-1}_{x}(y)\ld_{\ast}\varrho^{-1} L^{-1}_{x}(z)=
\varrho^{-1} L^{-1}_{\varrho^{-1} L^{-1}_{x}(y)}\varrho^{-1} L^{-1}_{x}(z)\stackrel{\scriptsize\eqref{eq:tau}}=\\
&
\varrho^{-1} L^{-1}_{\varrho^{-1}(x)}\varrho^{-1} L^{-1}_{L^{-1}_x(y)}(z)\stackrel{\scriptsize\eqref{eq:red}}=
\varrho^{-1} L^{-1}_{\varrho^{-1} (x)}\varrho^{-1} L^{-1}_y(z)\stackrel{\scriptsize\eqref{eq:tau}}=
\varrho^{-1} L^{-1}_{\varrho^{-1} (y)}\varrho^{-1} L^{-1}_x(z)\stackrel{\scriptsize\eqref{eq:red}}=\\
&\varrho^{-1} L^{-1}_{\varrho^{-1} (y)}\varrho^{-1} L^{-1}_{L^{-1}_y(x)}(z)
\stackrel{\scriptsize\eqref{eq:tau}}=\varrho^{-1} L^{-1}_{\varrho^{-1} L^{-1}_y(x)}\varrho^{-1} L^{-1}_y(z)=
\varrho^{-1} L^{-1}_{y}(x)\ld_{\ast}\varrho^{-1} L^{-1}_{y}(z)=\\
&(y\ld_{\ast} x)\ld_{\ast}(y\ld_{\ast} z),
\end{align*}
which shows that the left quasigroup $(X,\ast,\ld_{\ast})$ is right cyclic.
\end{proof}

\begin{example}
By Lemma \ref{lem:cycabel}, Condition \eqref{eq:sigma} is satisfied by every automorphism $\varrho$ of a $2$-reductive rack, since for $x,y\in X$
\begin{align*}
& L_{\varrho(y)}\varrho L_x=\varrho L_y\varrho^{-1} \varrho L_x=\varrho L_yL_x=\varrho L_xL_y=\varrho L_x\varrho^{-1} \varrho L_y=
L_{\varrho(x)} \varrho L_y.
\end{align*}
\end{example}

On the other hand, each $2$-permutational medial left quasigroup
has as an isotope that is a $2$-reductive rack.

\begin{theorem}\label{th:2perrozdz}
Let $(X,\circ,\ld_{\circ})$ be a left quasigroup and $\pi$ be a bijection on the set $X$ which satisfies Condition \eqref{eq:isoper1} and such that for each $x,y\in X$
\begin{align}
&L_x\pi L_y=L_y\pi L_x.\label{eq:isoper2}
\end{align}
Then the $\pi$-isotope of $(X,\circ,\ld_{\circ})$ is a $2$-reductive rack.
\end{theorem}

\begin{proof}
Let $(X,\ast,\ld_{\ast})$ be the $\pi$-isotope of $(X,\circ,\ld_{\circ})$. By Lemma \ref{lm:3cond}(1), $(X,\ast,\ld_{\ast})$ is $2$-reductive. Moreover, for $x,y\in X$ we have
\begin{align*}
L_{L_x\pi(y)}\pi L_x\stackrel{\scriptsize\eqref{eq:isoper1}}=
L_y\pi L_x\stackrel{\scriptsize\eqref{eq:isoper2}}=L_x\pi L_y.
\end{align*}
By Lemma \ref{lm:3cond}(3) the left quasigroup $(X,\ast,\ld_{\ast})$ is left distributive, and in consequence $2$-reductive rack.
\end{proof}

\begin{corollary}\label{cor:cor1}
Let $(X,\circ,\ld_{\circ})$ be a $2$-permutational medial left quasigroup and $e\in X$. Then the $L_e^{-1}$-isotope of $(X,\circ,\ld_{\circ})$ is a $2$-reductive rack.
\end{corollary}
\begin{proof}
By Lemmas \ref{lm:a5} and \ref{lm:a3}, for each $x,y\in X$
\begin{align*}
&L_{L_xL_e^{-1}(y)}=L_y\quad {\rm and}\quad L_xL_e^{-1}L_y=L_yL_e^{-1}L_x,
\end{align*}
which shows that Conditions \eqref{eq:isoper1} and \eqref{eq:isoper2} are satisfied for $\pi=L_e^{-1}$. Corollary follows by Theorem \ref{th:2perrozdz}.
\end{proof}

\begin{example}\label{ex:exp1}
Let $(X,\circ,\ld_{\circ})$  be the $2$-permutational medial left quasigroup from
Example \ref{ex:nondistr} and let $e=0$. Then  $(X,\ast,\ld_{\ast})$, with $x\ast y=x\circ L_0^{-1}(y)$ and
$x\ld_{\ast}y=L_0(x\ld_{\circ} y)$,  is a $2$-reductive rack with the $\ast$-multiplication table presented in Example \ref{ex:nie2red}.
\end{example}
The next example shows that the assumption of mediality in Corollary \ref{cor:cor1} is not always needed.
\begin{example}
Let $(\{0,1,2\},\circ,\ld_{\circ})$ be a left quasigroup with the following left multiplication:
 \[
  \begin{array}{c|ccc}
   \circ & 0 & 1 & 2 \\
   \hline
   0 & 0 & 2&1 \\
   1 & 0 & 2&1 \\
   2 & 1&2 & 0
  \end{array},
  \]
  i.e. $L_0=L_1=(12)$ and $L_2=(012)$.
This left quasigroup is $2$-permutational, but not medial
\[
0=0\circ 0=(0\circ 0)\circ (1\circ 0)\neq (0\circ 1)\circ (0\circ 0)=2\circ 0=1.
\]
But for $\pi=L_0^{-1}=(12)$ Condition \eqref{eq:isoper2} is satisfied and the $\pi$-isotope of $(\{0,1,2\},\circ,\ld_{\circ})$
 \[
  \begin{array}{c|ccc}
   \ast & 0 & 1 & 2 \\
   \hline
   0 & 0 & 1&2 \\
   1 & 0 & 1&2 \\
   2 & 1&0 & 2
  \end{array}
  \]
is $2$-reductive rack $((\Z_2,\Z_1),\left(\begin{smallmatrix}0&0\\1&0\end{smallmatrix}\right))$.
\end{example}

It is worth emphasizing that all results from Sections \ref{sec2} -- \ref{sec4a} established for left quasigroups are also true for right quasigroups, when using their dual versions.

\section{Left distributive involutive biracks}\label{sec3}

In the previous three sections we prepared tools that
we shall be now using on biracks -- universal algebraic
incarnations of set-theoretic solutions of the Yang-Baxter equation.
Originally, biracks are algebras studied in low-dimensional topology \cite{FJSK, EN}. The equational definition of a birack we use here was given first in \cite{S06}. (Note that Stanovsk\'y considered two left quasigroups there.)

\begin{de}
An algebra $(X,\circ,\ld_{\circ},\bullet,/_{\bullet})$ with four binary operations is called a {\em birack}, if $(X,\circ,\ld_{\circ})$ is a left quasigroup,
$(X,\bullet,/_{\bullet})$ is a right quasigroup and the following holds for any $x,y,z\in X$:
\begin{align}
x\circ(y\circ z)=(x\circ y)\circ((x\bullet y)\circ z),\label{eq:b1}\\
(x\circ y)\bullet((x\bullet y)\circ z)=(x\bullet(y\circ z))\circ(y\bullet z), \label{eq:b2}\\
(x\bullet y)\bullet z= (x\bullet (y\circ z))\bullet (y \bullet z).\label{eq:b3}
  \end{align}
\end{de}

We will say that a birack $(X,\circ,\ld_{\circ},\bullet,/_{\bullet})$ is \emph{left distributive}, if $(X,\circ,\ld_{\circ})$ is a rack,  is \emph{right distributive}, if for every $x,y,z \in X$
\begin{align*}
(y\bullet z)\bullet x=(y\bullet x)\bullet(z\bullet x),
\end{align*}
i.e. the right quasigroup $(X,\bullet,/_{\bullet})$ is \emph{right distributive}. The birack is \emph{distributive} if it is left and right distributive. It is evident that all properties of left distributive biracks stay true in its dual form for right distributive ones.

\begin{example}
Let $X$ be a non-empty set and let $f,g\colon X\to X$ be two bijections with $fg=gf$. An algebra $(X,\circ,\ld_{\circ},\bullet,/_{\bullet})$ such that for every $x,y\in X$,
\begin{align*}
x\circ y=f(y),\quad x\ld_{\circ} y=f^{-1}(y),\\
x\bullet y=g(x),\quad x/_{\bullet} y=g^{-1}(x)
\end{align*}
is a birack called $1$-\emph{permutational} (since
both quasigroups are $1$-permutational). If $f,g= {\rm id}$, $1$-permutational birack is called a \emph{projection birack}.
\vskip 2mm

Each $1$-permutational birack is left distributive, since for every $x,y,z\in X$
\begin{align*}
x\circ(y\circ z)=L_xL_y(z)=f^2(z)=L_{x\circ y}L_x(z)=(x\circ y)\circ(x\circ z).
\end{align*}
\end{example}
A birack is {\em idempotent} if both one-sided quasigroups $(X,\circ,\ld_{\circ})$ and $(X,\bullet,/_{\bullet})$ are idempotent.
And a birack is {\em involutive} if it additionally satisfies, for every $x,y\in X$:
 \begin{align}
 &(x\circ y)\circ(x\bullet y)=x, \label{eq:linv}\\
&  (x\circ y)\bullet(x\bullet y)=y.\label{eq:rinv}
 \end{align}

Note that Conditions \eqref{eq:linv} and  \eqref{eq:rinv} give, for every $x,y\in X$,
\begin{align}\label{eq:inv}
x\bullet y=L_{x\circ y}^{-1}(x)=(x\circ y)\ld_{\circ} x\quad {\rm and}\quad x\circ y={\textbf{\emph R}}^{-1}_{x\bullet y}(y)=y/_{\bullet}(x\bullet y).
 \end{align}
It follows then, that an involutive birack is idempotent if $(X,\circ,\ld_{\circ})$ \emph{or} $(X,\bullet,/_{\bullet})$ is idempotent. We shall see (Corollary \ref{stu})
that an involutive birack is left distributive if and only if it is right distributive.

The next, well known result (see \cite[Proposition 1]{Rump}, \cite[Proposition 1.5]{D15}, \cite[Section 4.2]{JPZ}) is crucial for our considerations.
\begin{theorem}\label{th:rclq}
An algebra $(X,\circ,\ld_{\circ},\bullet,/_{\bullet})$ is an involutive birack if and only if $(X,\circ,\ld_{\circ})$ is a non-degenerate right cyclic left quasigroup.
\end{theorem}
Recall, if $(X,\circ,\ld_{\circ})$ is a non-degenerate right cyclic left quasigroup then defining for every $x,y\in X$, $x\bullet y={\textbf{\emph R}}_y(x)=(x\circ y)\ld_{\circ} x$, and $x/_{\bullet} y={\textbf{\emph R}}^{-1}_y(x)$,
the algebra $(X,\circ,\ld_{\circ},\bullet,/_{\bullet})$ is an involutive birack.
\vskip 2mm

\begin{remark}
Conditions \eqref{eq:b1} -- \eqref{eq:b3} and \eqref{eq:linv} -- \eqref{eq:rinv} are dual with respect to operations $\circ$ and $\bullet$. Thus Theorem \ref{th:rclq} immediately implies (see \cite{D15}, \cite{Rump} or \cite[Section 4.2]{JPZ}) that in an involutive birack $(X,\circ,\ld_{\circ},\bullet,/_{\bullet})$, the right quasigroup $(X,\bullet,/_{\bullet})$ is non-degenerate and left cyclic i.e. for every $x,y,z\in X$
\begin{align*}
(z/_{\bullet} x)/_{\bullet}(y/_{\bullet} x)=(z/_{\bullet} y)/_{\bullet}(x/_{\bullet} y),
\end{align*}
and the mapping
\begin{align*}
S\colon X\to X; \quad x\mapsto x/_{\bullet}x
\end{align*}
is a bijection.

Moreover (see \cite{Rump} and \cite[Section 2]{JPZ}), operations $\ld_{\circ}$ and $/_{\bullet}$ are connected by
\begin{equation*}
(x\ld_{\circ} x)/_{\bullet}(x\ld_{\circ} x)=x\quad {\rm and}\quad
(x/_{\bullet} x)\ld_{\circ}(x/_{\bullet} x)=x,
\end{equation*}
which is equivalent to the fact that the mappings $S$ and $T\colon X\to X; x\mapsto x\ld_{\circ}x$ are mutually inverse. It simply means that each involutive birack is a \emph{biquandle} (see \cite{S06}).
\end{remark}

An involutive birack $(X,\circ,\ld_{\circ},\bullet,/_{\bullet})$ is $2$-\emph{reductive} if the left quasigroup $(X,\circ,\ld_{\circ})$ is $2$-reductive. By Theorem \ref{th:rclq} and Corollary \ref{cor:2red} we directly obtain the following.
\begin{corollary}\label{cor:eqvdis}
An involutive birack is left distributive if and only if it is $2$-reductive.
\end{corollary}

From now on, we will use both terms: a \emph{(left) distributive} involutive birack and a \emph{$2$-reductive} involutive birack, interchangeably.
\vskip 2mm
In some cases in a birack $(X,\circ,\ld_{\circ},\bullet,/_{\bullet})$, the left multiplication $\circ$
and the right multiplication $\bullet$ are mutually inverse, i.e. for every $x,y\in X$, the following  condition is satisfied:
\begin{equation}\label{eq:lri}
(x\circ y)\bullet x=y = x\circ(y\bullet x)\quad \Leftrightarrow \quad L_x=\emph{\textbf{R}}_x^{-1}.
\end{equation}
Condition \eqref{eq:lri} is called {\bf lri} (see \cite[Definition 2.18]{GIM08}).

For example, Condition {\bf lri} is satisfied in idempotent involutive biracks \cite[Corollary 2.33]{GIM08}.
Moreover, Gateva-Ivanova showed that also $2$-reductive involutive biracks satisfy this condition. Below we present a shorter alternative proof of this fact.
\begin{lemma}\label{lri2red}\cite[Lemma 7.1]{GI18}
An involutive $2$-reductive birack satisfies Condition {\bf lri}.
\end{lemma}
\begin{proof}
Let $(X,\circ,\ld_{\circ},\bullet,/_{\bullet})$ be an involutive $2$-reductive birack. Then, for each $x,y\in X$ we obtain
\[
x\circ(y\bullet x)\stackrel{\scriptsize \eqref{eq:inv}}=x\circ L^{-1}_{y\circ x}(y)\stackrel{\scriptsize \eqref{eq:red}}=x\circ L^{-1}_x(y)=x\circ(x\ld_{\circ} y)\stackrel{\scriptsize \eqref{eq:lq}}=y
\]
and
\[
(x\circ y)\bullet x\stackrel{\scriptsize \eqref{eq:inv}}=L^{-1}_{(x\circ y)\circ x}(x\circ y)\stackrel{\scriptsize \eqref{eq:red}}=L_x^{-1}L_x(y)=y. \qedhere
\]
\end{proof}

The converse statement to Lemma \ref{lri2red} is not true even for the idempotent case.

\begin{example}\label{ex:lri2red}
Let $(X=\{0,1,2,3,4\},\circ,\ld_{\circ},\bullet,/_{\bullet})$ be the following idempotent involutive birack: $L_0=L_3={\textbf{\emph R}}_0={\textbf{\emph R}}_3=(24)$, $L_1={\textbf{\emph R}}_1=(02)(34)$, $L_2=L_4={\textbf{\emph R}}_2={\textbf{\emph R}}_4=(03)$. The birack satisfies Condition {\bf lri}, but it is not $2$-reductive, since $L_4=L_{L_1(3)}\neq L_3$.
\end{example}

As a result of Lemma \ref{lri2red}, one obtains that an involutive left (or right) distributive birack is distributive.

\begin{corollary}\label{stu}

An involutive birack $(X,\circ,\ld_{\circ},\bullet,/_{\bullet})$ is left distributive if and only if it is right distributive.
\end{corollary}
\begin{proof}
By Corollary \ref{cor:eqvdis}, an involutive left distributive birack is $2$-reductive and by Lemma \ref{lri2red} it satisfies Condition {\bf lri}. Hence, for every $x,y\in X$, we have $x\bullet y=y\ld_{\circ} x$. By Lemma \ref{rem:eqival} $(X,\ld_{\circ},\circ)$ is left distributive, and straightforward calculations show that $(X,\bullet,/_{\bullet})$ is right distributive.

The proof in the opposite direction follows by the fact that a right distributive right quasigroup satisfies dual $2$-reductive law, and in consequence it also satisfies Condition {\bf lri}.
\end{proof}

\vskip 2mm
Moreover, if $(X,\circ,\ld_{\circ},\bullet,/_{\bullet})$ is an involutive distributive birack then
the left quasigroup $(X,\circ,\ld_{\circ})$ and the right quasigroup $(X,\bullet,/_{\bullet})$
are \emph{mutually orthogonal}, i.e. for every $a,b\in X$,
the pair of equations
\begin{align*}
a=x\circ y\;\; {\rm and}\;\; b=x\bullet y
\end{align*}
has a unique solution: $x=a\circ b$ and $y=b\ld_{\circ} a$. Indeed, by Corollary \ref{cor:eqvdis}

the left quasigroup $(X,\circ,\ld_{\circ})$ is $2$-reductive. Therefore, we have
\begin{align*}
&x\circ y=(a\circ b)\circ (b\ld_{\circ} a)=L_{a\circ b}L_b^{-1}(a)\stackrel{\scriptsize \eqref{eq:red}}=L_bL_b^{-1}(a)=a.
\end{align*}
Further, by Lemma \ref{lri2red},
$x\bullet y=y\ld_{\circ} x$ and $x/_{\bullet}y=y\circ x$. Hence,
\begin{align*}
&x\bullet y=(a\circ b)\bullet (b\ld_{\circ} a)=(b\ld_{\circ} a)\ld_{\circ}(a\circ b)=L^{-1}_{b\ld_{\circ} a}L_a(b)\stackrel{\scriptsize \eqref{eq:red}}=L_a^{-1}L_a(b)=b.
\end{align*}

Since by Corollary \ref{cor:eqvdis}, for an involutive distributive birack $(X,\circ,\ld_{\circ},\bullet,/_{\bullet})$, the left quasigroup $(X,\circ,\ld_{\circ})$ is $2$-reductive,
Theorem \ref{th:tri} immediately implies
\begin{theorem}\label{th5.8}
Each involutive distributive birack $(X,\circ,\ld_{\circ},\bullet,/_{\bullet})$ is a disjoint union, over a set $I$, of abelian groups
$A_j=\left<\{a_{i,j}\mid i\in I\}\right>$, for every $j\in I$, with operations:
\begin{align*}
&x\circ y=y+a_{i,j}\quad {\rm and}\quad x\ld_{\circ} y=y-a_{i,j},\\
&x\bullet y=y\ld_{\circ} x=x-a_{j,i} \quad {\rm and}\quad x/_{\bullet} y=x+a_{j,i},
\end{align*}
for $x\in A_i$ and $y\in A_j$.
\end{theorem}
Taking the notion from $2$-reductive racks, we will shortly say that the birack $(X,\circ,\ld_{\circ},\bullet,/_{\bullet})$ is the sum of a trivial affine mesh $\mathcal A=((A_i)_{i\in I},\,(a_{i,j})_{i,j\in I})$ over a set $I$. Note that each orbit is a $1$-permutational birack.
\vskip 2mm

Recall Condition $(\ast)$ discussed in Section \ref{sec2}.
Involutive distributive biracks without fixed points are examples of biracks which do not satisfy the condition.
The representation of involutive distributive birack as the sum of a trivial affine mesh allows one to verify quickly Condition $(\ast)$.
\begin{rem}\label{rem:star}
Let $(X,\circ,\ld_{\circ},\bullet,/_{\bullet})$ be an involutive distributive birack. By Theorem \ref{th5.8}, the birack $(X,\circ,\ld_{\circ},\bullet,/_{\bullet})$ is the sum of a trivial affine mesh $((A_i)_{i\in I},\,(c_{i,j})_{i,j\in I})$ over a set $I$. Then $(X,\circ,\ld_{\circ},\bullet,/_{\bullet})$ satisfies Condition $(\ast)$ if and only if
\[
\forall i\in I\quad\forall x\in A_i\quad\exists j\in I\quad \exists a\in A_j\quad a\circ x=x+c_{j,i}=x
\quad \Leftrightarrow\quad
\forall i\in I\quad\exists j\in I\quad c_{j,i}=0.
\]
\end{rem}
Remark \ref{rem:star} says that an involutive distributive birack satisfies Condition $(\ast)$ if and only if in each column in the matrix of constants there is at least one $0$.
\begin{example}
Let a birack $(X,\circ,\ld_{\circ},\bullet,/_{\bullet})$ be the sum of the trivial affine mesh $((\Z_4,\Z_4),\left(\begin{smallmatrix}1&2\\2&1\end{smallmatrix}\right))$. Then $(X,\circ,\ld_{\circ},\bullet,/_{\bullet})$ is distributive but does not satisfy Condition $(\ast)$. This birack is also not $1$-permutational.
\end{example}

\vskip 2mm

Let $(X,\circ,\ld_{\circ},\bullet,/_{\bullet})$ be a birack. Etingof, Schedler and Soloviev defined in \cite{ESS} the relation
\begin{align}\label{rel:sim}
a\sim b \quad \Leftrightarrow\quad L_a=L_b \quad \Leftrightarrow\quad \forall x\in X\quad a\circ x=b\circ x.
\end{align}

By their results, the relation $\sim$ is a congruence of involutive biracks, i.e. an equivalence relation on the set $X$ preserving all four operations in a birack $(X,\circ,\ld_{\circ},\bullet,/_{\bullet})$.

\begin{example}
Let $(X,\circ,\ld_{\circ},\bullet,/_{\bullet})$ be an involutive distributive birack. By Theorem \ref{th5.8}, the birack $(X,\circ,\ld_{\circ},\bullet,/_{\bullet})$ is the sum of a trivial affine mesh $((A_i)_{i\in I},\,(c_{i,j})_{i,j\in I})$ over a set $I$. For $a\in A_i$ and $b\in A_j$
\[
a\sim b \quad \Leftrightarrow\quad \forall k\in I\quad\forall x\in A_k\quad x+c_{i,k}=a\circ x=b\circ x=x+c_{j,k}
\quad \Leftrightarrow\quad \forall k\in I\quad c_{i,k}=c_{j,k}.
\]
\end{example}

\begin{lemma}\label{lm:sf}
Let $(X,\circ,\ld_{\circ},\bullet,/_{\bullet})$ be an involutive distributive birack. Then the quotient birack $(X/\mathord{\sim},\circ,\ld_{\circ},\bullet,/_{\bullet})$ is a projection one. 
\end{lemma}
\begin{proof}
By Corollary \ref{cor:eqvdis},
$(X,\circ,\ld_{\circ},\bullet,/_{\bullet})$ is $2$-reductive. In consequence, $x\circ y\sim y$. Since the relation $\sim$ is a congruence of $(X,\circ,\ld_{\circ},\bullet,/_{\bullet})$, $x\sim x$ and $(x\circ y)\bullet x\sim y\bullet x$. By Lemma \ref{lri2red},
$(x\circ y)\bullet x=y$ which gives $y\sim y\bullet x$.
\end{proof}

\section{2-permutational involutive biracks}\label{sec4}
Lemma \ref{lm:sf} shows that, for each involutive distributive birack, its quotient by the relation \eqref{rel:sim}
is a projection birack. There are also not distributive involutive biracks such that the quotient is a $1$-permutational birack.
\begin{example}
 Let $(X=\{0,1,2,3\},\circ,\ld_{\circ},\bullet,/_{\bullet})$ be the following involutive birack:
 \[
  \begin{array}{c|cccc}
   \circ & 0 & 1 & 2 & 3\\
   \hline
   0 & 1 & 0 & 3 &2\\
   1 & 3 & 2 & 1 & 0 \\
   2 & 1 & 0 & 3 & 2\\
   3 & 3 & 2 & 1 &0
  \end{array}
  \qquad
  \begin{array}{c|cccc}
  \bullet & 0 & 1 & 2 & 3\\
   \hline
   0 & 3 & 1 & 3 &1\\
   1 & 2 & 0 & 2 & 0 \\
   2 & 1 & 3 & 1 & 3\\
   3 & 0 & 2 & 0 &2
  \end{array},
 \]
 i.e. $L_0=L_2={\textbf{\emph R}}_1={\textbf{\emph R}}_3=(01)(23)$ and $L_1=L_3={\textbf{\emph R}}_0={\textbf{\emph R}}_2=(03)(12)$.
Example \ref{ex:nondistr} shows that the birack  $(X,\circ,\ld_{\circ},\bullet,/_{\bullet})$ is not left distributive, but the left quasigroup
$(X,\circ,\ld_{\circ})$ is $2$-permutational. Clearly, the quotient $(X/\mathord{\sim},\circ,\ld_{\circ},\bullet,/_{\bullet})$
\[
  \begin{array}{c|cc}
   \circ & 0/\mathord{\sim} & 1/\mathord{\sim}\\
   \hline
   0/\mathord{\sim} & 1/\mathord{\sim} & 0/\mathord{\sim} \\
   1/\mathord{\sim} & 1/\mathord{\sim} & 0/\mathord{\sim}
  \end{array}
  \qquad
  \begin{array}{c|cc}
  \bullet & 0/\mathord{\sim} & 1/\mathord{\sim}\\
   \hline
   0/\mathord{\sim} & 1/\mathord{\sim} & 1/\mathord{\sim} \\
   1/\mathord{\sim} & 0/\mathord{\sim} & 0/\mathord{\sim}
  \end{array}.
\]
is a $1$-permutational, but not a projection birack.
\end{example}

\begin{de}
An involutive birack $(X,\circ,\ld_{\circ},\bullet,/_{\bullet})$ is $2$-\emph{permutational} (\emph{medial}) if the left quasigroup $(X,\circ,\ld_{\circ})$ is $2$-permutational ({medial}).
\end{de}

\begin{proposition}\label{lm:birmed}
An involutive birack is $2$-permutational if and only if it is medial.
\end{proposition}
\begin{proof}
Let $(X,\circ,\ld_{\circ},\bullet,/_{\bullet})$ be an involutive $2$-permutational birack.
Since the relation \eqref{rel:sim} is a congruence of an involutive birack then by \eqref{eq:lq} and \eqref{eq:2per} for every $x,y,z\in X$ we have:
\[
z\ld_{\circ} y=z\ld_{\circ} (x\circ (x\ld_{\circ} y))\sim z\ld_{\circ} (z\circ (x\ld_{\circ} y))=x\ld_{\circ} y,
\]
which implies
\[L_{x\ld_{\circ} y}=L_{z\ld_{\circ} y}.\] By Theorem \ref{th:rclq}, the left quasigroup $(X,\circ,\ld_{\circ})$ is right cyclic. Hence for $x,y,a,b\in X$ we obtain
\[
L_xL_{a\ld_{\circ} y}=L_xL_{x\ld_{\circ} y}\stackrel{\scriptsize \eqref{eq:RC}}=L_yL_{y\ld_{\circ} x}=L_yL_{b\ld_{\circ} x}.
\]
Substitution of $x$ by $b\circ x$ and $y$ by $a\circ y$ gives that the birack $(X,\circ,\ld_{\circ},\bullet,/_{\bullet})$ is medial
\[
L_{b\circ x}L_y=L_{a\circ y}L_x\stackrel{\scriptsize \eqref{eq:2per}}=L_{b\circ y}L_x.
\]
Lemma \ref{lm:medial-2perm} completes the proof.
\end{proof}

Rump showed in \cite[Theorem 2]{Rump} that each finite right cyclic left quasigroup is non-degenerate (see also \cite[Proposition 4.7]{JPZ}). Therefore, directly by Theorem \ref{th:rclq}  and Proposition \ref{lm:birmed}, we obtain
\begin{corollary}\label{cor:med}
Each finite $2$-permutational right cyclic left quasigroup is medial.
\end{corollary}
But the following question is still open.
\begin{ques}
Is it true that every infinite $2$-permutational right cyclic left quasigroup is medial?
\end{ques}

\begin{corollary}\label{cor:ldidemp}
An involutive $2$-permutational birack $(X,\circ,\ld_{\circ},\bullet,/_{\bullet})$ is distributive if and only if the quotient $(X/\mathord{\sim},\circ,\ld_{\circ},\bullet,/_{\bullet})$ is idempotent.
\end{corollary}
\begin{proof}
By Proposition \ref{lm:birmed} the birack  $(X,\circ,\ld_{\circ},\bullet,/_{\bullet})$ is medial.
Let $(X/\mathord{\sim},\circ,\ld_{\circ},\bullet,/_{\bullet})$ be idempotent. This implies that for each $x\in X$, $x\sim x\circ x$.
Therefore, by Lemma \ref{lm:a3}, for every $e,x\in X$,
\[
L_{e\circ x}=L_{e\circ e}L_xL_e^{-1}=L_{e}L_xL_e^{-1}.
\]
Lemma \ref{lm:sf} completes the proof.
\end{proof}

Using Condition {\bf lri} it is easy to recognize distributive biracks among $2$-permutational involutive ones.
\begin{lemma}\label{c:lri2red}
Let $(X,\circ,\ld_{\circ},\bullet,/_{\bullet})$ be a $2$-permutational involutive birack.
Then $(X,\circ,\ld_{\circ},\bullet,/_{\bullet})$ is distributive if and only if it satisfies Condition {\bf lri}.
\end{lemma}
\begin{proof}
Let  $(X,\circ,\ld_{\circ},\bullet,/_{\bullet})$ be $2$-permutational and let it satisfy~{\bf lri}. Then
\[ (x\circ y)\circ z  \stackrel{\scriptsize {\bf lri}}= (x\circ y) \circ ( (y\circ z) \bullet y) \stackrel{\scriptsize \eqref{eq:2per}}= ( (y\circ z) \circ y ) \circ ( (y\circ z ) \bullet y ) \stackrel{\scriptsize \eqref{eq:linv}}= y\circ z. \]  
The converse follows by Lemma \ref{lri2red}.
\end{proof}

The condition of $2$-permutationality in Lemma \ref{c:lri2red} cannot be weakened, even in the idempotent case, as we see on the next example (see also Example \ref{ex:lri2red}).
\begin{example}
Let $(X=\{0,1,2,3\},\circ,\ld_{\circ},\bullet,/_{\bullet})$ be the following involutive birack: $L_0={\textbf{\emph R}}_0=(02)$, $L_1=L_3={\textbf{\emph R}}_1={\textbf{\emph R}}_3=\id$, $L_2={\textbf{\emph R}}_2=(02)(13)$. Clearly, the birack satisfies Condition {\bf lri}, but it is not $2$-permutational, since $L_2=L_{L_1(2)}\neq L_{L_0(2)}=L_0$.
\end{example}

In Section~4 we presented the notion of a $\pi$-isotope.
This construction allows us to tie distributive and 2-permutational
biracks.

Let $(X,\circ,\ld_{\circ},\bullet,/_{\bullet})$ be an involutive birack. By Theorem \ref{th:rclq}, $(X,\circ,\ld_{\circ})$ is a right-cyclic, non-degenerate left quasigroup. Let $\pi$ be a bijection of a set $X$ such that the $\pi$-isotope $(X,\ast,\ld_{\ast})$ of $(X,\circ,\ld_{\circ})$ is right cyclic. Then, by Remark \ref{rem:nd} and Theorem \ref{th:rclq}, one can define uniquely the involutive birack $(X,\ast,\ld_{\ast},\diamond,/_{\diamond})$. We will call the birack obtained in this way the \emph{$\pi$-isotope} of $(X,\circ,\ld_{\circ},\bullet,/_{\bullet})$. Note that then
\begin{align*}
&x\ast y:=L_x\pi (y)\quad {\rm and} ,\\
&x\diamond y=(x\ast y)\ld_{\ast}x=\pi^{-1}L_{L_x\pi(y)}^{-1}(x)=\pi^{-1}(x\bullet\pi(y)).
\end{align*}

\begin{remark}\label{th:distrzper1}
Let $(X,\circ,\ld_{\circ},\bullet,/_{\bullet})$  be an involutive birack and let $\pi$ be a bijection on the set $X$ which satisfies Conditions \eqref{eq:isoper1} and \eqref{eq:isoper2}. By Theorem \ref{th:2perrozdz}, the $\pi$-isotope $(X,\ast,\ld_{\ast})$ of $(X,\circ,\ld_{\circ})$ is a $2$-reductive rack and by Lemma \ref{lm:2red} it is right cyclic. The $\pi$-isotope $(X,\ast,\ld_{\ast},\diamond,/_{\diamond})$ of $(X,\circ,\ld_{\circ},\bullet,/_{\bullet})$ is a distributive involutive birack. Moreover, by Lemma \ref{lri2red}, $(X,\ast,\ld_{\ast},\diamond,/_{\diamond})$ satisfies Condition {\bf lri}, i.e. $x\diamond y=y\ld_{\ast} x=\pi^{-1}L_y^{-1}(x)$ and $x/_{\diamond}y=y\ast x=L_y\pi(x)$. This can be also obtained by direct calculations with use of Condition \eqref{eq:isoper1}.
\end{remark}

Let $(X,\ast,\ld_{\ast},\diamond,/_{\diamond})$ be the $\pi$-isotope of a finite involutive birack $(X,\circ,\ld_{\circ},\bullet,/_{\bullet})$ and let $(X,\ast,\ld_{\ast},\diamond,/_{\diamond})$ satisfy Condition {\bf lri}. Then the multiplication table of $\ast$ is obtained by a permuting columns of the multiplication
table of $\circ$ and
the multiplication table of $\diamond$ is obtained by a permuting rows of the multiplication table of $\bullet$.
\begin{remark}\label{th:distrzper2}
Let $(X,\circ,\ld_{\circ},\bullet,/_{\bullet})$ be an involutive distributive birack and let $\pi$ be a bijection on the set $X$ which satisfies Condition \eqref{eq:sigma}. By Corollary \ref{cor:eqvdis} and Theorem \ref{th:rozdz2per}, the $\pi$-isotope $(X,\ast,\ld_{\ast})$ of $(X,\circ,\ld_{\circ})$ is a $2$-permutational right cyclic left quasigroup. Then the $\pi$-isotope $(X,\ast,\ld_{\ast},\diamond,/_{\diamond})$ of $(X,\circ,\ld_{\circ},\bullet,/_{\bullet})$ is a $2$-permutational involutive birack.
\end{remark}

Note that by Lemmas \ref{lm:a5} and \ref{lm:a3} for a 2-permutational involutive birack $(X,\circ,\ld_{\circ},\bullet,/_{\bullet})$ it is always possible to construct its \emph{non-trivial} $\pi$-isotope, taking $\pi=L_{e}^{-1}\neq\id$, for any $e\in X$.
\begin{lemma}\label{lm:2p_iso}
Let $(X,\circ,\ld_{\circ},\bullet,/_{\bullet})$ be a $2$-permutational involutive birack and let $e\in X$. The $L_{e}^{-1}$-isotope of $(X,\circ,\ld_{\circ},\bullet,/_{\bullet})$ is an involutive distributive birack.
\end{lemma}
\begin{proof}
By Proposition \ref{lm:birmed}, the birack $(X,\circ,\ld_{\circ},\bullet,/_{\bullet})$ is medial. By Corollary \ref{cor:cor1} the $L_e^{-1}$-isotope $(X,\ast,\ld_{\ast})$ of $(X,\circ,\ld_{\circ})$ is a 2-reductive rack. Then the $L_e^{-1}$-isotope $(X,\ast,\ld_{\ast},\diamond,/_{\diamond})$ of $(X,\circ,\ld_{\circ},\bullet,/_{\bullet})$ is a distributive involutive birack.
\end{proof}

Theorem below shows that each $2$-permutational involutive birack originates from an involutive distributive birack.

\begin{theorem}\label{th:repres}
Each $2$-permutational involutive birack is a $\pi$-isotope of a distributive one, for some bijection $\pi$.
\end{theorem}

\begin{proof}

Let $(X,\circ,\ld_{\circ},\bullet,/_{\bullet})$ be a $2$-permutational involutive birack and let $e\in X$. By Lemma \ref{lm:2p_iso} the $L_e^{-1}$-isotope $(X,\ast,\ld_{\ast},\diamond,/_{\diamond})$ of $(X,\circ,\ld_{\circ},\bullet,/_{\bullet})$ is a distributive involutive birack.

Let $\varrho=L_e$. By Lemma \ref{lm:a3}(3) we have for each $x,y,z\in X$
\begin{align*}
&\varrho(y)\ast \varrho(x\ast z)=L_{L_e(y)}L_e^{-1}L_eL_xL_e^{-1}(z)=L_{L_e(y)}L_xL_e^{-1}(z)=L_{L_e(e)}L_yL_e^{-1}L_xL_e^{-1}(z)=\\
&L_{L_e(e)}L_xL_e^{-1}L_yL_e^{-1}(z)=L_{L_e(x)}L_yL_e^{-1}(z)
=L_{L_e(x)}L_e^{-1}L_eL_yL_e^{-1}(z)=\varrho(x)\ast\varrho(y\ast z),
\end{align*}
which shows that the left quasigroup $(X,\ast,\ld_{\ast})$ satisfies Condition \eqref{eq:sigma}, for $\varrho=L_e$.

Moreover, for each $x,y\in X$
\begin{align*}
&x\ast \varrho(y)=L_xL_e^{-1}L_e(y)= x\circ y \quad {\rm and} \quad \varrho^{-1}(x\ld_{\ast}y)=L_e^{-1}L_eL_x^{-1}(y)=x\ld_{\circ}y,
\end{align*}
which shows that $(X,\circ,\ld_{\circ},\bullet,/_{\bullet})$ is the $L_e$-isotope of the involutive distributive birack $(X,\ast,\ld_{\ast},\diamond,/_{\diamond})$.
\end{proof}
Now we collect some useful facts about bijections satisfying Conditions \eqref{eq:isoper1} and \eqref{eq:sigma}.
\begin{remark}\label{rm.rho}
Let $(X,\circ,\ld_{\circ},\bullet,/_{\bullet})$ be an involutive birack and let $\varrho$ be a bijection on the set $X$ which satisfies Condition \eqref{eq:sigma}. Then, for every $x,y\in X$,
\begin{align*}
x\sim y\quad \Leftrightarrow \quad \varrho(x)\sim\varrho(y).
\end{align*}
Indeed, by Definition \ref{rel:sim} we have
\begin{align*}
x\sim y\ \Leftrightarrow\ L_x=L_y\ \Rightarrow\ L_{\varrho(y)}\varrho L_x\stackrel{\scriptsize \eqref{eq:sigma}}=L_{\varrho(x)}\varrho L_y=
L_{\varrho(x)}\varrho L_x \ \Rightarrow\ L_{\varrho(x)}=L_{\varrho(y)}\ \Leftrightarrow\ \varrho(x)\sim\varrho(y).
\end{align*}
On the other hand,
\begin{align*}
\varrho(x)\sim\varrho(y) \ \Leftrightarrow\ L_{\varrho(x)}=L_{\varrho(y)}\ \Rightarrow\
L_{\varrho(y)}\varrho L_x \stackrel{\scriptsize \eqref{eq:sigma}}=L_{\varrho(x)}\varrho L_y=L_{\varrho(y)}\varrho L_y\ \Rightarrow\
L_x=L_y  \ \Leftrightarrow\ x\sim y.
\end{align*}

\end{remark}

\begin{remark}\label{rem:2red2red}
Let $(X,\circ,\ld_{\circ},\bullet,/_{\bullet})$ be an involutive birack which satisfies Condition {\bf lri}. Consider the $\pi$-isotope $(X,\ast,\ld_{\ast},\diamond,/_{\diamond})$ of $(X,\circ,\ld_{\circ},\bullet,/_{\bullet})$ for some bijection $\pi$ of the set $X$. Then $(X,\ast,\ld_{\ast},\diamond,/_{\diamond})$ satisfies Condition {\bf lri} if and only if for every $x\in X$
\begin{align*}
L_{\pi(x)}=L_x.\label{eq:pi2red}
\end{align*}
Indeed, for every $x,y\in X$
\begin{align*}
x\diamond y=y\ld_{\ast} x \quad \Leftrightarrow\quad \pi^{-1}(x\bullet \pi(y))=\pi^{-1}(x\bullet y)\quad \Leftrightarrow\quad {\textbf{\emph R}}_{\pi(y)}={\textbf{\emph R}}_y\quad \Leftrightarrow\quad L_{\pi(y)}^{-1}=L_y^{-1}.
\end{align*}
If $\pi$ happens to be an automorphism of the left quasigroup $(X,\circ,\ld_{\circ})$, then
\begin{align*}
L_x=L_{\pi(x)}=\pi L_x\pi^{-1}\quad \Leftrightarrow\quad \pi L_x=L_x\pi.
\end{align*}
Hence, in this case the $\pi$-isotope of $(X,\circ,\ld_{\circ},\bullet,/_{\bullet})$ satisfies then Condition {\bf lri} if and only if the automorphism $\pi$ commutes with each left translation.
In particular, the $\pi$-isotope of an involutive 2-reductive birack is 2-reductive for any choice $\pi=L_e$ or $\pi=L_e^{-1}$, with $e\in X$.
\end{remark}

\begin{example}
Let $(\{0,1,2,3\},\circ,\ld_{\circ},\bullet,/_{\bullet})$  be the $2$-permutational involutive birack with its left quasigroup $(\{0,1,2,3\},\circ,\ld_{\circ})$ from
Example \ref{ex:nondistr}. Then the
$L_0^{-1}$-isotope $(\{0,1,2,3\},\ast,\ld_{\ast},\diamond,/_{\diamond})$, with $x\ast y=x\circ L_0^{-1}(y)$,
$x\ld_{\ast}y=0\circ L_x^{-1}(y)$, $x\diamond y=0\circ L_y^{-1}(x)$ and $x/_{\diamond}y:=y\circ L_0^{-1}(x)$ is an involutive distributive birack with the $\ast$-table presented in Example \ref{ex:nie2red}.
\end{example}

\begin{example}\label{ex:new}
Let $(\{0,1,2,3\},\circ,\ld_{\circ},\bullet,/_{\bullet})$ be the distributive involutive birack with the left quasigroup $(\{0,1,2,3\},\circ,\ld_{\circ})$ defined in Example~\ref{ex:nie2red}. Note that the permutation $\pi=(01)(23)$ satisfies Condition \eqref{eq:sigma}. Then constructing the $\pi$-isotope of $(\{0,1,2,3\},\circ,\ld_{\circ},\bullet,/_{\bullet})$ we obtain the 2-permutational involutive birack $(\{0,1,2,3\},\ast,\ld_{\ast},\diamond,/_{\diamond})$ with the $\ast$-table presented in Example~\ref{ex:nondistr}. The $\diamond$-table is the following
\[
  \begin{array}{c|cccc}
   \diamond & 0 & 1 & 2 & 3\\
   \hline
   0 & 3 & 1 & 3 &1\\
   1 & 2 & 0 & 2 & 0 \\
   2 & 1 & 3 & 1 & 3\\
   3 & 0 & 2 & 0 &2
  \end{array}.
 \]
Since $L_0=L_2={\textbf{\emph R}}_1={\textbf{\emph R}}_3=(01)(23)$ and $L_1=L_3={\textbf{\emph R}}_0={\textbf{\emph R}}_2=(03)(12)$, the birack $(\{0,1,2,3\},\ast,\ld_{\ast},\diamond,/_{\diamond})$ does not satisfy Condition {\bf lri}.
\end{example}

\vskip 2mm

Note that different choices of a bijection in the construction of the isotope may give non isomorphic biracks.
\begin{example}\label{ex:6.10}
Let $(\{0,1,2,3,4\},\circ,\ld_{\circ},\bullet,/_{\bullet})$ be the $2$-permutational involutive birack with multiplication $\circ$
 \[
  \begin{array}{c|ccccc}
   \circ & 0 & 1 & 2 & 3 &4\\
   \hline
   0 & 0& 2 & 1 & 4 &3\\
   1 & 3 & 2 & 1 & 0 &4\\
   2 & 4 & 2 & 1 & 3&0\\
   3 & 0& 2 & 1 & 4 &3\\
   4&0& 2 & 1 & 4 &3
  \end{array},
 \]
 i.e. $L_0=L_3=L_4=(12)(34)$, $L_1=(03)(12)$ and $L_2=(04)(12)$. Then
$L_i^{-1}$-isotopes, for $i\in\{0,1\}$, of $(\{0,1,2,3,4\},\circ,\ld_{\circ},\bullet,/_{\bullet})$ have the following multiplication tables of~$\ast_i$
 \[
  \begin{array}{c|ccccc}
   \ast_0 & 0 & 1 & 2 & 3 &4\\
   \hline
   0 & 0& 1 & 2 & 3 &4\\
   1 & 3 & 1 & 2 & 4 &0\\
   2 & 4 & 1 & 2 & 0&3\\
   3 & 0& 1 & 2 & 3 &4\\
   4&0& 1 & 2 & 3 &4
  \end{array} \qquad{\rm and}\quad
  \begin{array}{c|ccccc}
   \ast_1 & 0 & 1 & 2 & 3 &4\\
   \hline
   0 & 4& 1 & 2 & 0 &3\\
   1 & 0 & 1 & 2 & 3 &4\\
   2 & 3 & 1 & 2 & 4&0\\
   3 & 4& 1 & 2 & 0 &3\\
   4 & 4& 1 & 2 & 0 &3
  \end{array}.
 \]
Both isotopes are distributive.
It is clear that these two biracks are not isomorphic,
as the $L_0^{-1}$-isotope is idempotent, whereas the $L_1^{-1}$-isotope is not.
\end{example}


\begin{example}\label{ex:612}
In Example \ref{ex:new} we showed that the birack $(\{0,1,2,3\},\ast,\ld_{\ast},\diamond,/_{\diamond})$ with the $\ast$-table presented in Example~\ref{ex:nondistr} is the $\pi$-isotope, for $\pi=(01)(23)$, of the distributive birack $(X,\circ,\ld_{\circ},\bullet,/_{\bullet})$ with the left quasigroup $(X,\circ,\ld_{\circ})$ defined in Example~\ref{ex:nie2red}.
Nevertheless, there is another choice of a permutation that yields another birack. If we take $\gamma=(0123)$ then this $\gamma$ satisfies Condition \eqref{eq:sigma} as well
and we obtain the involutive 2-permutational birack with multiplication $\ast_1$:
\[
  \begin{array}{c|cccc}
   \ast_1 & 0 & 1 & 2 & 3\\
   \hline
   0 & 1 & 2 & 3 &0\\
   1 & 3 & 0 & 1 & 2 \\
   2 & 1 & 2 & 3 & 0\\
   3 & 3 & 0 & 1 &2
  \end{array}  \qquad
  \text{ or in other words }  \qquad
  \begin{array}{r@{\>=\>}l}
   L_0=L_2={\textbf{\emph R}}_0={\textbf{\emph R}}_2&(0123)\\
   L_1=L_3={\textbf{\emph R}}_1={\textbf{\emph R}}_3&(3210)
  \end{array}
  \]
  which is clearly not isomorphic to the birack $(\{0,1,2,3\},\ast,\ld_{\ast},\diamond,/_{\diamond})$.
Note that both permutations $\pi$ and $\gamma$ are actually
automorphisms of the birack $(X,\circ,\ld_{\circ},\bullet,/_{\bullet})$ but neither the $\pi$-isotope nor the $\gamma$-isotope is isomorphic to $(X,\circ,\ld_{\circ},\bullet,/_{\bullet})$.
\end{example}

For an involutive birack $(X,\circ,\ld_{\circ},\bullet,/_{\bullet})$, a bijection $h$ on the set $X$ is an isomorphism between the $\alpha$-isotope and the $\beta$-isotope of $(X,\circ,\ld_{\circ},\bullet,/_{\bullet})$ if and only if
\begin{align}\label{eq:isoiso}
hL_x\alpha(y)=h(x\circ \alpha(y))=h(x)\circ\beta h(y)=L_{h(x)}\beta h(y).
\end{align}
Hence, we obtain the following observation.
\begin{remark}
Let $(X,\circ,\ld_{\circ},\bullet,/_{\bullet})$ be an involutive birack. An automorphism $h$ of $(X,\circ,\ld_{\circ})$ is an isomorphism between the $\alpha$-isotope and the $\beta$-isotope of $(X,\circ,\ld_{\circ},\bullet,/_{\bullet})$ if and only if
\begin{align*}
\alpha=h^{-1}\beta h.
\end{align*}
\end{remark}

\section{Solutions}\label{sec:sol}
As it was written in Section~1, each solution $(X,\sigma,\tau)$ of the Yang-Baxter equation yields an involutive birack $(X,\circ,\ld_{\circ},\bullet,/_{\bullet})$. And conversely,
if $(X,\circ,\ld_{\circ},\bullet,/_{\bullet})$ is an involutive birack, then defining
\begin{align*}
r(x,y)=(\sigma(x,y),\tau(x,y))=(x\circ y,x\bullet y)=(L_x(y),{\textbf{\emph R}}_y(x)),
\end{align*}
we obtain a solution $(X,L,{\textbf{\emph R}})$ of the Yang-Baxter equation.

Such an equivalence allows us to treat each solution as an involutive birack and formulate results from Sections \ref{sec3} and \ref{sec4} in the language of solutions. In particular, $1$-permutational birack corresponds to a \emph{permutation solution} and the projection birack corresponds to the \emph{trivial solution}.
\vskip 2mm

Etingof et al. reasoned that the quotient set $X/\mathord{\sim}$, by the relation \eqref{rel:sim}, has a structure of a solution  $(X/\mathord{\sim},\overline{\sigma},\overline{\tau})$  with $\overline{\sigma}(x/\mathord{\sim},y/\mathord{\sim})=\sigma(x,y)/\mathord{\sim}$ and $\overline{\tau}(x/\mathord{\sim},y/\mathord{\sim})=\tau(x,y)/\mathord{\sim}$ for $x/\mathord{\sim},y/\mathord{\sim}\in X/\mathord{\sim}$  and $x\in x/\mathord{\sim}, y\in y/\mathord{\sim}$. They called such solution the \emph{retraction} of $(X,\sigma,\tau)$ and denoted it by ${\rm Ret}(X,\sigma,\tau)$. The birack corresponding to the retraction solution ${\rm Ret}(X,\sigma,\tau)$ is the quotient birack $(X/\mathord{\sim},\circ,\ld_{\circ},\bullet,/_{\bullet})$.

Among solutions, an important role is played by \emph{multipermutation solutions}, see e.g. \cite{CJO10,GIC12,Ven}.
Let $(X,\sigma,\tau)$ be a solution. One defines \emph{iterated retraction} in the following way: ${\rm Ret}^0(X,\sigma,\tau):=(X,\sigma,\tau)$ and
${\rm Ret}^k(X,\sigma,\tau):={\rm Ret}({\rm Ret}^{k-1}(X,\sigma,\tau))$, for any natural number $k>1$.
A solution $(X,\sigma,\tau)$ is called a \emph{multipermutation solution of level $m$} if $m$ is the least nonnegative integer such that
\[
|{\rm Ret}^{m}(X,\sigma,\tau)|=1.
\]

In the language of an involutive birack $(X,\circ,\ld_{\circ},\bullet,/_{\bullet})$ this means that applying $m$ times the congruence $\sim$ to the subsequent quotient biracks, one obtains the one-element birack.

Let us consider $(X/\mathord{\sim},\circ,\ld_{\circ},\bullet,/_{\bullet})$, the quotient birack of $(X,\circ,\ld_{\circ},\bullet,/_{\bullet})$  and denote it by  ${\rm Ret}(X,\circ,\bullet)$.
Let ${\rm Ret}^0(X,\circ,\bullet):=(X,\circ,\bullet)$ and
${\rm Ret}^k(X,\circ,\bullet):={\rm Ret}({\rm Ret}^{k-1}(X,\circ,\bullet))$, for any natural number $k>1$.

\begin{de}
An involutive birack is a \emph{multipermutation birack} if there exists a positive integer $m$ such that ${\rm Ret}^{m-1}(X,\circ,\bullet)$ is a $1$-permutational birack.
A birack of $(X,\circ,\ld_{\circ},\bullet,/_{\bullet})$ is called a \emph{multipermutation birack of level $m$} if $m$ is the least nonnegative integer $m$ such that
\[
|{\rm Ret}^{m}(X,\circ,\bullet)|=1.
\]
\end{de}
A birack $(X,\circ,\ld_{\circ},\bullet,/_{\bullet})$ is \emph{irretractable} if ${\rm Ret}(X,\circ,\bullet):=(X,\circ,\bullet)$, i.e.
$\sim$ is the trivial relation.

\begin{obser}\cite[Section 3]{GIM07}
Let $|X|\geq 2$. A square-free solution $(X,\sigma,\tau)$ is a multipermutation solution of level $m$ if and only if
${\rm Ret}^{m-1}(X,\sigma,\tau)$ is a trivial solution.
\end{obser}

\begin{theorem}\cite[Proposition 4.7]{GI18}\label{GI:mper}
Let $(X,\sigma,\tau)$ be a solution and $|X|\geq 2$.
$(X,\sigma,\tau)$ is a multipermutation solution of level $0\leq m$ if and only if Condition \eqref{eq:mper} holds for the corresponding birack $(X,\circ,\ld_{\circ},\bullet,/_{\bullet})$.
\end{theorem}

\begin{de}
A solution is \emph{distributive} ($2$-\emph{reductive}, $2$-\emph{permutational}, \emph{medial}, respectively), if it corresponds to a distributive ($2$-reductive, $2$-permutational, medial, respectively) involutive birack.
\end{de}

\begin{fact}\cite[Proposition 8.2]{GIC12}, \cite[Proposition 4.7]{GI18}
A square-free solution $(X,\sigma,\tau)$ is multipermutation of level $2$ if and only if it is distributive. In this case it has an abelian permutation group $\left\langle \sigma_x\colon x\in X\right\rangle$.
More generally, if a solution satisfies Condition $(\ast)$ then it is a multipermutation solution of level $2$ if and only if it is $2$-reductive.
\end{fact}

By Corollary \ref{cor:eqvdis}, Proposition \ref{lm:birmed}, Corollary \ref{cor:ldidemp} and  Theorem \ref{GI:mper},
we can generalize some results given in \cite{GIM08, GIC12, GI18}.
\begin{theorem}\label{th:medial-char}
Let $(X,\sigma,\tau)$  be a solution. Then
\begin{enumerate}
\item $(X,\sigma,\tau)$ is a multipermutation solution of level $2$ if and only if it is medial.
\item If $(X,\sigma,\tau)$ is distributive then it is a multipermutation solution of level~$2$.
\end{enumerate}
\end{theorem}

\begin{theorem}\label{th:dist-char}
Let $(X,\sigma,\tau)$  be a multipermutation solution of level $2$. The following conditions are equivalent:
\begin{enumerate}
\item $(X,\sigma,\tau)$ is distributive,
\item $(X,\sigma,\tau)$ is $2$-reductive,
\item $(X,\sigma,\tau)$ satisfies Condition {\bf lri}, i.e. $\forall x\in X$ $\tau_x=\sigma_x^{-1}$,
\item ${\rm Ret}(X,\sigma,\tau)$ is the trivial solution.
\end{enumerate}
\end{theorem}

By Theorem \ref{th5.8} we can completely describe all distributive solutions.

\begin{theorem}
Each distributive solution $(X,\sigma,\tau)$ is a disjoint union, over a set $I$, of abelian groups
$A_j=\left<\{a_{i,j}\mid i\in I\}\right>$, for every $j\in I$, with
\begin{equation}\label{eq:7.8}
\sigma_x(y)=y+a_{i,j}\quad {\rm and} \quad \tau_y(x)= x-a_{j,i},
\end{equation}
where $x\in A_i$ and $y\in A_j$.
\end{theorem}
By Corollary \ref{stu} each distributive solution satisfies Condition {\bf stu}, introduced in \cite[Definition 5.1]{GIM08}, which means that it is trivially a strong twisted union of abelian groups $A_j$.

\begin{example}
Let $I$ be a (finite or infinite) index set
and let $A_i$, for $i\in I$, be cyclic groups.
Let $(a_{i,j})_{i,j\in I}$ be constants
such that $a_{i,j}\in A_j$, for all~$i,j\in I$,
and, for each~$j\in I$, there exists at least one~$i\in I$,
such that $a_{i,j}$ is a generator of the group $A_j$.
Then $(\bigcup A_i,\sigma,\tau)$, with $\sigma$ and $\tau$
defined in \eqref{eq:7.8}, is a distributive solution.

\end{example}

We can construct all distributive solutions of~size~$n$ using
the following algorithm:
\begin{algorithm}
Outputs all distributive solutions of size~$n$:
\begin{enumerate}
 \item For all partitionings $n=n_1+n_2+\cdots +n_k$ do (2)--(4).
 \item For all abelian groups $A_1$, \dots ,$A_k$ of size $|A_i|=n_i$ do (3)--(4).
 \item For all constants $a_{i,j}\in A_j$ do (4).
 \item If, for all $1\leq j\leq k$, we have $A_j=\langle \{a_{i,j}\mid i\in I\}\rangle$
 then construct a solution $(\bigcup A_i,\sigma,\tau)$ using \eqref{eq:7.8}.
\end{enumerate}
\end{algorithm}
When all solutions are constructed, we can get rid of isomorphic
copies using Theorem~\ref{th:iso}.
\vskip 2mm

In \cite{CJR} the permutation group $\left\langle \sigma_x\colon x\in X\right\rangle$ of a \emph{finite} solution $(X,\sigma,\tau)$ was called \emph{the involutive Yang-Baxter group} (IYB group) associated to the solution $(X,\sigma,\tau)$. In particular, Cedo et al. showed in \cite[Corollary 3.11]{CJR} that each finite nilpotent group of class 2 (and thus each finite abelian group) is an IYB group. Here, using the construction of the sum of a trivial affine mesh, we present short direct proof of this fact for an abelian group of an arbitrary cardinality.
\begin{theorem}\label{thm7.11}
 Let $A$ be an abelian group. Then there exists a solution
 $(X,\sigma,\tau)$ with its permutation group isomorphic to
 $A$.
\end{theorem}

\begin{proof}
 Let $A$ be generated by a (finite or infinite) set
 $\{g_i\colon\ i\in I\}$. We construct the solution~$(X,\sigma,\tau)$ as
 the sum of the trivial affine mesh $((A_i)_{i\in I},\,(c_{i,j})_{i,j\in I})$ over~$I$,
 with $A_i=A$ and $c_{i,j}=g_i$, for all~$i,j\in I$.

 By construction, $L_a(b)=b+c_{i,j}=b+g_i$, for all $a\in A_i$
 and $b\in A_j$. Therefore the permutation group
 consists solely of mappings $b_j\mapsto b_j+c$, for each $j\in I$
 and some~$c\in A$. This means that the
 group $\left\langle \sigma_x\colon x\in X\right\rangle$ naturally  embeds into~$A$.
 Moreover, the permutation group is generated by $L_a$, for $a\in A$, and hence
 it is isomorphic to~$A$.
\end{proof}

By Corollary \ref{cor:cor1} and Lemma \ref{lm:2p_iso} each multipermutation solution of level 2 defines a distributive one.

\begin{theorem}\label{th:2persol}
Let $(X,\sigma,\tau)$ be a multipermutation solution of level~$2$
and $e\in X$. Then $(X,\sigma',\tau')$, where $\sigma_x'=\sigma_x\sigma_e^{-1}$ and $\tau_y'=\sigma_e\tau_{\sigma_e^{-1}(y)}$, for $x,y\in X$, is a distributive solution.
\end{theorem}

On the other side, Theorem \ref{th:repres} shows that each multipermutation solution of level 2 originates from a distributive solution. We have even more. The theorem gives a procedure how to obtain all multipermutation solutions of level~$2$ from distributive ones.

\vskip 2mm

We have to take all distributive solutions $(X,\sigma,\tau)$ such that there exists $a\in X$ with $L_a={\rm id}$ and, for each of them, all permutations
$\pi$ of the set $X$ which satisfy Condition \eqref{eq:sigma} i.e. for $x,y\in X$
\[
\sigma_{\pi(y)}\pi\sigma_x=\sigma_{\pi(x)}\pi\sigma_y.
\]
Then $(X,\sigma',\tau')$, where $\sigma_x'=\sigma_x\pi$ and $\tau_y'=\pi^{-1}\tau_{\pi(y)}$, will be multipermutation solutions of level~$2$.
\vskip 2mm

By Lemma \ref{c:lri2red} and Remark \ref{rm.rho} we can construct all non-distributive solutions
of multipermutation level~2 of size~$n$.

\begin{algorithm} Outputs all non-distributive solutions
of multipermutation level~2 of size~$n$:
 \begin{enumerate}
  \item For every distributive solution $(X,\sigma,\tau)$ of size~$n$ do (2)--(7).
  \item If there exist no $x\in X$ such that $\sigma_x=\mathrm{id}$ return to (1).
  \item For every permutation $\pi\in S_X$ do (4)--(6).
  \item If $\sigma_x^{-1}=\tau_{\pi(x)}$, return to (3).
  \item If $\pi$ does not send classes of~$\sim$ onto classes of $\sim$, return to (3).
  \item If $\pi$ does not satisfy \eqref{eq:sigma}, return to (3).
 \item Construct the solution $(X,\sigma',\tau')$, where $\sigma'(x,y)=\sigma(x,\pi(y))$ and $\tau'(x,y)=\pi^{-1}(\tau(x,\pi(y)))$.
 \end{enumerate}
\end{algorithm}

Unlike in the case of distributive solutions, we do not
have any efficient criterion to test isomorphisms.
As Example~\ref{ex:6.10} shows, the same solutions can
be obtained from different distributive solutions.

\section{Enumeration}\label{sec:enum}
In this section we enumerate solutions of multipermutation level~2
for small sizes and we estimate, for all sizes, how many racks and solutions are there, up to isomorphism.

Using the characterization by sums of trivial affine meshes we can straightforwardly describe all solutions of
small sizes. The size 4 can be done manually.
\begin{example}
By results of \cite{ESS}, there are 23 solutions of size 4, up to isomorphism. Two of them are irretractable.
Exactly 17 of them are distributive. They are the sums of the following trivial affine meshes:
\begin{itemize}
    \item One orbit: $((\Z_4),(1)).$

        \item Two orbits:
        $((\Z_3,\Z_1),\left(\begin{smallmatrix}0&0\\1&0\end{smallmatrix}\right))$,
        $((\Z_3,\Z_1),\left(\begin{smallmatrix}1&0\\0&0\end{smallmatrix}\right))$,
        $((\Z_3,\Z_1),\left(\begin{smallmatrix}1&0\\1&0\end{smallmatrix}\right))$,
        $((\Z_3,\Z_1),\left(\begin{smallmatrix}2&0\\1&0\end{smallmatrix}\right))$,\\
        $((\Z_2,\Z_2),\left(\begin{smallmatrix}0&0\\1&1\end{smallmatrix}\right))$,
        $((\Z_2,\Z_2),\left(\begin{smallmatrix}1&1\\1&0\end{smallmatrix}\right))$, $((\Z_2,\Z_2),\left(\begin{smallmatrix}1&0\\1&1\end{smallmatrix}\right))$,
        $((\Z_2,\Z_2),\left(\begin{smallmatrix}1&0\\0&1\end{smallmatrix}\right))$,\\
        $((\Z_2,\Z_2),\left(\begin{smallmatrix}0&1\\1&0\end{smallmatrix}\right))$,
        $((\Z_2,\Z_2),\left(\begin{smallmatrix}1&1\\1&1\end{smallmatrix}\right))$.

         \item Three orbits: $((\Z_2,\Z_1,\Z_1),\,\left(\begin{smallmatrix}1&0&0\\0&0&0\\0&0&0\end{smallmatrix}\right))$,
         $((\Z_2,\Z_1,\Z_1),\,\left(\begin{smallmatrix}0&0&0\\0&0&0\\1&0&0\end{smallmatrix}\right))$,
         $((\Z_2,\Z_1,\Z_1),\,\left(\begin{smallmatrix}0&0&0\\1&0&0\\1&0&0\end{smallmatrix}\right))$,\\
         $((\Z_2,\Z_1,\Z_1),\,\left(\begin{smallmatrix}1&0&0\\0&0&0\\1&0&0\end{smallmatrix}\right))$,
         $((\Z_2,\Z_1,\Z_1),\,\left(\begin{smallmatrix}1&0&0\\1&0&0\\1&0&0\end{smallmatrix}\right))$.

        \item Four orbits: $((\Z_1,\Z_1,\Z_1,\Z_1),\,\left(\begin{smallmatrix}0&0&0&0\\0&0&0&0\\0&0&0&0\\0&0&0&0\end{smallmatrix}\right))$.
\end{itemize}
There remain four multipermutation solutions of size~$4$
that are not distributive. Two of them are of level~$2$,
both of them described in Example~\ref{ex:612}. Two of them are
of level~$3$ - the corresponding biracks have the following tables of $\circ$-multiplication:
\[
\begin{array}{c|cccc}
    \circ & 0 & 1 & 2 & 3\\
   \hline
   0 & 0 & 1 & 2 &3\\
   1 & 0 & 1 & 2 & 3 \\
   2 & 0 & 1 & 3 & 2\\
   3 & 1 & 0 & 3 &2
  \end{array} \qquad {\rm and}\quad
  \begin{array}{c|cccc}
    \circ & 0 & 1 & 2 & 3\\
   \hline
   0 & 1 & 0 & 2 &3\\
   1 & 1 & 0 & 2 & 3 \\
   2 & 0 & 1 & 3 & 2\\
   3 & 1 & 0 & 3 &2
  \end{array}.
  \]
  They are not isomorphic since the first one has two idempotent elements, whereas the other one has none.
\end{example}

The same way as we did it for size~$4$, we can compute
other small sizes, on a computer of course. We start with
the numbers of small racks.
In Table \ref{Fig:count_medial}, we compare the numbers of isomorphism classes of all racks (see OEIS sequence A181770 \cite{OEIS}) and 2-reductive racks.
Computing $2$-reductive racks directly using Theorems~\ref{th:tri} and~\ref{th:iso} is hopeless for larger numbers.
Hence the numbers of $2$-reductive racks were computed
using Burnside's lemma, see~\cite{JPSZ} for more details.

As we can see, the numbers of $2$-reductive racks grow really fast,
actually, according to Blackburn~\cite{B}, there are at least $2^{n^2/4-O(n\log n)}$ $2$-reductive racks of size~$n$.
We can also give an upper bound, which is not far from the lower bound. The proof is exactly the same as the proof of~\cite[Theorem 8.2]{JPSZ}.

\begin{theorem}\label{Th:Count-rack}
 There are at most $2^{(1/4+o(1))n^2}$ $2$-reductive racks of size~$n$, up to isomorphism.
\end{theorem}

The numbers in Table~\ref{Fig:count_medial} suggest that the vast majority of all racks are $2$-reductive.
However, we do not have any proof of this fact.
Hence we can only conjecture that:

\begin{conjecture}
 There are $2^{(1/4+o(1))n^2}$ racks of size~$n$, up to isomorphism.
\end{conjecture}

In Table~\ref{Fig:count_solution}, we can see the numbers of solutions.
The total numbers of solutions is taken from~\cite{ESS}, the numbers of $2$-reductive
solutions are the same as the numbers of $2$-reductive racks.
The numbers of $2$-permutational solutions (i.e. multipermutation solutions of level~$2$) that are not $2$-reductive
were computed by a brute force search algorithm using the Mace4 software \cite{Prover}.

\begin{theorem}\label{Th:Count-solution}
 There are at most $2^{(1/4+o(1))n^2}$ multipermutation solutions of level~$2$ of size~$n$,
 up to isomorphism.
\end{theorem}

\begin{proof}
 As was shown in Theorem~\ref{th:2persol}, every multipermutation solution of level~$2$ can be obtained
 as an isotope of a $2$-reductive solution using a permutation.
 Hence, using Theorem~\ref{Th:Count-rack}, the number of $2$-permutational solutions
 is less than $2^{(1/4+o(1))n^2}\cdot n!=2^{(1/4+o(1))n^2}$.
\end{proof}

In the case of solutions, Table~\ref{Fig:count_solution} suggests that the numbers of all solutions
grow faster than the numbers of multipermutation solutions of level~$2$ but not much faster.
We can therefore conjecture:

\begin{conjecture}
 There are $2^{O(n^2)}$ solutions of size~$n$, up to isomorphism.
\end{conjecture}

\begin{table}
\begin{small}
$$\begin{array}{|r|rrrrrrrrrrr|}\hline
n                       & 1& 2& 3& 4&  5&  6&   7&    8&     9&    10&      11\\\hline
\text{racks}           & 1& 2& 6& 19& 74& 353& 2080& 16023& & & \\
\text{2-reductive}    & 1& 2& 5& 17&  65&  323&   1960&    15421&    155889&     2064688&       35982357 \\ \hline
\end{array}$$
$$\begin{array}{|r|rrr|}\hline
n                       &      12&        13&     14 \\\hline
\text{racks}           & & &  \\
\text{2-reductive racks}    &      832698007 &        25731050861&    1067863092309\\ \hline
\end{array}$$
\end{small}
\caption{The number of racks and 2-reductive racks of size $n$, up to isomorphism.}
\label{Fig:count_medial}
\end{table}

\begin{table}
\begin{small}
$$\begin{array}{|r|rrrrrrrr|}\hline
n                       & 1& 2& 3& 4&  5&  6&   7&    8\\\hline
\text{involutive solutions}    & 1& 2& 5& 23&  88&  595&   3456&    34528 \\ \hline
\text{multipermutation of level~$2$}    & 1& 2& 5& 19&  70&  359&   2095&    16332 \\ \hline
\text{$2$-reductive}    & 1& 2& 5& 17&  65&  323&   1960&    15421 \\ \hline
\text{$2$-permutational, not $2$-reductive}    & 0& 0& 0& 2&  5&  36&   135&    911 \\ \hline
\end{array}$$
\end{small}
\caption{The number of involutive solutions of size $n$, up to isomorphism.}
\label{Fig:count_solution}
\end{table}

\end{document}